\documentclass[12pt]{article}
\usepackage{amsmath}
\usepackage{amssymb}
\usepackage{amsthm}
\usepackage{euscript}
\usepackage{hyperref}

\title{Transseries: Ratios,\\Grids, and Witnesses}

\author{G. A. Edgar}

\date{\today}

\theoremstyle{plain}
\newtheorem{pr}{Proposition}
\newtheorem{thm}[pr]{Theorem}
\newtheorem{co}[pr]{Corollary}
\newtheorem{lem}[pr]{Lemma}

\theoremstyle{remark}
\newtheorem{re}[pr]{Remark}
\newtheorem{de}[pr]{Definition}
\newtheorem{no}[pr]{Notation}
\newtheorem{ex}[pr]{Example}

\newtheorem{qu}[pr]{Question}

\numberwithin{pr}{section}

\newcommand{\Def}[1]{\textbf{\itshape #1}} 

\newcommand{\Esubgrid}{Prop.~3.35}

\newcommand{\Eheightwins}{Prop.~3.72}

\newcommand{\EcomponN}{Prop.~3.98}

\newcommand{\Ederivpropertiesa}{Prop.~3.114(a)}

\newcommand{\Edominates}{Def.~4.12}
\newcommand{\Emudominateprop}{Prop.~4.17}
\newcommand{\Ecostinfixed}{Prop.~4.22}

\newcommand{\Cexponentiality}{Prop.~4.5}

\newcommand{\Cderivcompare}{Prop.~4.10}
\newcommand{\Cmvti}{Prop.~4.12}
\newcommand{\Cinverse}{Prop.~4.20}
\newcommand{\Ctaylorlabel}{5.1}
\newcommand{\Cconverge}{Sec.~6}

\renewcommand{\phi}{\varphi}
\renewcommand{\epsilon}{\varepsilon}
\renewcommand{\emptyset}{\varnothing}
\newcommand{\takes}{\colon}
\newcommand{\fgt}{\succ}
\newcommand{\fst}{\prec}
\newcommand{\fe}{\asymp}

\newcommand{\fgteq}{\succcurlyeq}
\newcommand{\fsteq}{\preccurlyeq}

\newcommand{\SET}[2]{ \left\{\, {#1} : {#2} \,\right\} }
\newcommand{\R}{\mathbb R}

\newcommand{\N}{\mathbb N}
\newcommand{\Z}{\mathbb Z}

\newcommand{\G}{\mathfrak G}
\renewcommand{\AA}{\mathfrak A}
\newcommand{\BB}{\mathfrak B}
\newcommand{\MM}{\mathfrak M}
\newcommand{\GRID}{\mathfrak J}
\newcommand{\SA}{\EuScript A}
\newcommand{\SD}{\EuScript D}

\newcommand{\LP}{\EuScript P}
\renewcommand{\P}{\EuScript P}

\newcommand{\SW}{\EuScript W}

\newcommand{\BF}{\mathbf F}

\newcommand{\T}{\mathbb T}
\newcommand{\TW}[1]{{}^{#1}\T}
\newcommand{\TWG}[2]{{}^{#1}\T^{#2}}

\newcommand{\bk}{\mathbf k}
\newcommand{\bp}{\mathbf p}
\newcommand{\bq}{\mathbf q}
\newcommand{\bm}{\mathbf m}
\newcommand{\bn}{\mathbf n}
\newcommand{\0}{\mathbf 0}
\newcommand{\fa}{\mathfrak a}
\newcommand{\fb}{\mathfrak b}

\newcommand{\g}{\mathfrak g}
\newcommand{\m}{\mathfrak m}
\newcommand{\n}{\mathfrak n}

\renewcommand{\l}{\mathfrak l}
\newcommand{\A}{\mathbf A}
\newcommand{\B}{\mathbf B}

\newcommand{\DD}{\mathbf D}

\newcommand{\Gsmall}{\G^{\mathrm{small}} }
\newcommand{\Glarge}{\G^{\mathrm{large}} }
\newcommand{\supp}{\operatorname{supp}}
\newcommand{\lsupp}{\operatorname{lsupp}}
\newcommand{\tsupp}{\operatorname{tsupp}}
\renewcommand{\mag}{\operatorname{mag}}

\newcommand{\Min}{\operatorname{Min}}
\newcommand{\dom}{\operatorname{dom}}

\newcommand{\sm}{\operatorname{small}}
\newcommand{\la}{\operatorname{large}}
\newcommand{\const}{\operatorname{const}}

\newcommand{\bmu}{{\boldsymbol{\mu}}}
\newcommand{\ebmu}{{\boldsymbol{\mu}}}
\newcommand{\tbmu}{{\widetilde{\bmu}}}
\newcommand{\tebmu}{{\tilde{\ebmu}}}
\newcommand{\ba}{{\boldsymbol{\alpha}}}
\newcommand{\bb}{{\boldsymbol{\beta}}}

\newcommand{\bnu}{{\boldsymbol{\nu}}}

\newlength{\uuu}
\newcommand{\lbb}{\begin{picture}(8,8)(-2,2)
	\put(0,0){\line(0,1){9}}
	\put(2,0){\line(0,1){9}}
	\put(0,0){\line(1,0){5}}
	\put(0,9){\line(1,0){5}}
	\end{picture}}
\newcommand{\rbb}{\begin{picture}(7,8)(0,2)
	\put(3,0){\line(0,1){9}}
	\put(5,0){\line(0,1){9}}
	\put(0,0){\line(1,0){5}}
	\put(0,9){\line(1,0){5}}
	\end{picture}}
\newcommand{\lbbb}{\begin{picture}(10,8)(-2,2)
	\put(0,0){\line(0,1){9}}
	\put(2,0){\line(0,1){9}}
	\put(4,0){\line(0,1){9}}
	\put(0,0){\line(1,0){7}}
	\put(0,9){\line(1,0){7}}
	\end{picture}}
\newcommand{\rbbb}{\begin{picture}(9,8)(0,2)
	\put(3,0){\line(0,1){9}}
	\put(5,0){\line(0,1){9}}
	\put(7,0){\line(0,1){9}}
	\put(0,0){\line(1,0){7}}
	\put(0,9){\line(1,0){7}}
	\end{picture}}

\begin{document} 
\maketitle
\setcounter{tocdepth}{4}

\abstract{More remarks and questions on transseries.  In particular
we deal with the system of ratio sets and grids used in the
grid-based formulation of transseries.  This involves
a ``witness'' concept that keeps track of the ratios required
for each computation.  There are, at this stage, questions and
missing proofs in the development.}

\tableofcontents


\section{Introduction}
Most of the definitions and computations with transseries
found in \cite{edgar} (see ``Review'' below)
were done in the ``grid-based'' setting.
But often the use of the ratio set was just a hint or an aside.
Here we will carry out these constructions more completely.

This is also an attempt to derive results in a manner
continuing the elementary approach of \cite{edgar}.  So in
some cases I am attempting alternate proofs for results that already
exist in the literature.

I am using the totally ordered monomial group $\G$.  Maybe
there should be separate consideration of the parts that are
valid for partially ordered (or quasi-ordered) monomial group.
This would be useful if (when) we have to discuss $\R\lbbb x,y \rbbb$.

\subsection*{Review}
The differential field $\T$ of transseries is completely explained in
my recent expository introduction \cite{edgar}.  Other
sources for the definitions are:
\cite{asch}, \cite{costintop}, \cite{DMM}, \cite{hoeven}.
I will generally follow
the notation from \cite{edgar}.
Write $\LP = \SET{S \in \T}{S \fgt 1, S > 0}$ for the set of
large positive transseries. The operation of
composition $T \circ S$ is defined for $T \in \T$, $S \in \LP$.
The set $\LP$ is a group under composition
(\cite[\S~5.4.1]{hoeven}, \cite[Cor.~6.25]{DMM},
\cite[\Cinverse]{edgarc}.
Both notations $T\circ S$ and $T(S)$ will be used.

We write $\G$ for the ordered group of transmonomials.
We write $\G_{N,M}$ for the transmonomials with
exponential height $N$ and logarithmic depth $M$.
We write $\G_N$ for the log-free transmonomials with height $N$.
Let $\l_m = \log\log \cdots \log x$ with $m$ logarithms.
A ratio set $\bmu$ is a finite subset of $\Gsmall$;
$\GRID^\ebmu$ is the group generated by $\bmu$.
If $\bmu = \{\mu_1,\cdots,\mu_n\}$, then
$\GRID^\ebmu = \SET{\bmu^\bk}{\bk \in \Z^n}$.
If $\bm \in \Z^n$, then
$\GRID^{\ebmu,\bm} = \SET{\bmu^\bk}{\bk \in \Z^n, \bk \ge \bm}$
is a grid.  A grid-based transseries is supported
by some grid.  A subgrid is a subset of a grid.
If $T \in \T = \R\lbb\G\rbb$, then
the support $\supp T$ is a subgrid.

Recall \cite{edgar} some of the reasons for using
the grid-based field $\R\lbb\MM\rbb$ instead of the
full well-based Hahn field $\R[[\MM]]$:
\begin{enumerate}
\item[(i)] The finite ratio set is conducive to computer
calculations.
\item[(ii)] Problems from analysis almost always
have solutions in this smaller system.
\item[(iii)] Some proofs and formulations of definitions
are simpler in one system than in the other.
\item[(iv)] Perhaps (?) the analysis used for \'Ecalle--Borel
convergence can be applied only to grid-based series.
\item[(v)] In the well-based case, the domain of $\exp$
cannot be all of $\R[[\MM]]$.
\item[(vi)] The grid-based ordered set $\R\lbb \MM \rbb$ is a
``Borel order,'' but the well-based ordered set $\R[[\MM]]$ is not.
\end{enumerate}

\section{Framework}
When $A = L + c + S$ with $L$ purely large, $c \in \R$, $S$ small,
write $L = \la A$, $c = \const A$, and $S = \sm A$.
For $\AA,\BB \subseteq \G$, write
$\AA\BB = \SET{\fa\fb}{\fa\in \AA, \fb\in \BB}$.  And for $\g \in \G$, write
$\g \BB := \{\g\}\BB = \SET{\g\fb}{\fb\in \BB}$.

\begin{re}
For $\g \in \G$ and $A \in \T$, we have
$\supp(\g A) = \g\supp A$.  But for $A,B \in \T$,
we have only
$\supp(A B) \subseteq \supp A \supp B$, and not necessarily
equality, because of possible cancellation.  If all coefficients
are $\ge 0$ then there is no cancellation.
\end{re}

Let $\bmu = \{\mu_1,\cdots,\mu_n\} \subset \Gsmall$ be a ratio set.
Write $\bmu^*$ for the set of words and $\bmu^+$ the set
of nonempty words over $\bmu$ (the monoid and semigroup,
respectively, generated by $\bmu$).  That is,
$$
	\bmu^* = \SET{\bmu^\bk}{\bk \in \Z^n, \bk \ge \0},
	\qquad
	\bmu^+ = \SET{\bmu^\bk}{\bk \in \Z^n, \bk > \0}.
$$
Empty-set conventions say: $\emptyset^* = \{1\}$ and
$\emptyset^+ = \emptyset$.
The grids $\GRID^{\ebmu,\bm}$ may then be written
$\GRID^{\ebmu,\bm} = \bmu^\bm\bmu^*$.
In \cite{hoeven}, the definition of \Def{grid} is
more general: a set of the form $\g\bmu^*$.  But in this totally ordered
setting, we have the following.

\begin{pr}\label{griddef}
Let $\g \in \G$ and let $\ba \subset \Gsmall$ be finite.
Then there is finite $\bmu \subset \Gsmall$
and $\bm \in \Z^n$ such that $\g\ba^* \subseteq \GRID^{\ebmu,\bm}$.
\end{pr}

So, \Def{subgrid} (that is, a subset of some grid) has the same meaning for
each of the two definitions of ``grid.''

\begin{proof}[Proof of Proposition~\ref{griddef}]
If $\g = 1$, let $\bmu = \ba$ and $\bm = 0$ so that $\bmu^\bm = \g$.
If $\g \fst 1$,
let $\bmu = \ba \cup \{\g\}$, and let $\bm$ have a single nonzero component
$1$ so that $\bmu^\bm = \g$.  If $\g \fgt 1$,
let $\bmu = \ba \cup \{\g^{-1}\}$, and let $\bm$ have a single nonzero
component $-1$ so that $\bmu^\bm = \g$.
\end{proof}

When he allows a partially ordered $\G$, van der Hoeven \cite{hoeven} defines
a \Def{grid} as a finite
union of sets of the form $\g\bmu^*$.  But
in our (totally ordered) case
that is taken care of by the following.

\begin{pr}
Given any two grids, there is a third grid that contains them both.
\end{pr}
\begin{proof}
(\cite[Lemma~7.8]{DMM}.)
Let $\GRID^{\bmu,\bm}$ and $\GRID^{\bnu,\bn}$ be grids.
If we define $\ba = \bmu \cup \bnu$ and extend $\bm$ and $\bn$
with $0$s in the new components, the two grids are contained
respectively in $\GRID^{\ba,\bm}$ and $\GRID^{\ba,\bn}$.
Then, let $\bp$ be the componentwise minimum of $\bm$ and $\bn$,
so that both of these grids are contained in $\GRID^{\ba,\bp}$.
\end{proof}

The partial well order property of $\Z^n$ is
used for the next result.  This result turns out to be very useful.
It looks simple (and it is), but it is essential for the
theory.  (The name will be explained below.)

\begin{thm}[Subgrid Witness Theorem]\label{sgwt}
Let $\AA \subseteq \GRID^{\ebmu,\bm}$ be a nonempty subgrid.
Let $\g = \max \AA$.  Then there is a ratio set
$\ba \subset \GRID^\ebmu$
such that $\AA \subseteq \g\ba^*$.
\end{thm}
\begin{proof}
(See \cite[4.198]{costinasymptotics},
\cite[Lemma~7.8]{DMM},
\cite[Proposition~2.1]{hoeven}.)
Let $\BF := \SET{\bk \in \Z^n}{\bk \ge \bm, \bmu^\bk \in \AA}$.  Then
the set $\Min \BF$ of minimal elements of $\BF$ is finite.
Now $\g = \max\AA$ so $\g = \bmu^\bp$ for some $\bp \in \Min \BF$.  Let
$$
	\ba := \bmu \cup \SET{\bmu^\bk/\g}{\bk \in \Min F, \bmu^\bk \ne \g} .
$$
[So $\ba$ consists of $\bmu$ together with a finite number
of additional monomials, all elements of $\GRID^\ebmu$.]
We claim $\AA \subseteq \g\ba^*$.
Indeed, let $\n \in \AA$, say $\n = \bmu^\bn$ where $\bn \in \BF$.
Then there is $\bk \in \Min \BF$ so that $\bk \le \bn$.
Now $\bmu^\bk/\g \in \ba$, so $\bmu^\bk \in \g\ba \subseteq \g\ba^*$.
And $\n/\bmu^\bk = \bmu^{\bn-\bk} \in \bmu^* \subseteq \ba^*$.  So
$\n \in \g\ba^*\ba^* = \g\ba^*$.
\end{proof}

\subsection*{Order, Far Larger}
Let $\bmu \subset \Gsmall$ be a (finite) ratio set.
Let $\m, \n \in \G$.  Then:
$$
	\m \fsteq^\ebmu \n \Longleftrightarrow \m/\n \in \bmu^* ,
	\qquad
	\m \fst^\ebmu \n \Longleftrightarrow \m/\n \in \bmu^+ .
$$
We may rephrase this:
$$
	\m \fsteq^\ebmu \n \Longleftrightarrow \m \in \n\bmu^* ,
	\qquad
	\m \fst^\ebmu \n \Longleftrightarrow \m \in \n\bmu^+ .
$$
Of course $\m \fst \n$ if and only if there exists $\bmu$
such that $\m \fst^\ebmu \n$.

Let $\AA,\BB \subseteq \G$ be two sets.
Then \cite[\Edominates]{edgar} we say $\AA \fst^\ebmu \BB$ iff:
for every $\fa \in \AA$ there exists $\fb \in \BB$ with $\fa\fst^\ebmu \fb$,
and we say $\BB$ \Def{$\bmu$-dominates} $\AA$.  So
\begin{equation*}
	\AA \fst^\ebmu \BB \Longleftrightarrow \AA \subseteq \BB\,\bmu^+ .
\end{equation*}
Similarly, define:
\begin{align*}
	\AA \fsteq^\ebmu \BB \Longleftrightarrow \AA \subseteq \BB\,\bmu^* ,
	\qquad
	\AA \fe^\ebmu \BB \Longleftrightarrow \AA\,\bmu^* = \BB\,\bmu^* .
\end{align*}

The corresponding non-generator definition could be: $\AA \fst \BB$ iff
for every $\fa \in \AA$ there exists $\fb \in \BB$ with $\fa\fst \fb$.
Of course $\AA \fst^\ebmu \BB \Longrightarrow \AA \fst \BB$.  But:

\begin{ex}
Let $\AA=\{x^{-1/2}, x^{-2/3}, x^{-3/4}, x^{-4/5},\cdots\}$,
$\BB = \{1\}$.  Then $\AA \fst \BB$, but
there is no finite $\bmu \subset \Gsmall$
such that $\AA \fst^\ebmu \BB$.
\end{ex}

Let $A,B \in \T$.  Then we say
$A \fst^\ebmu B$ iff $\supp A \fst^\ebmu \supp B$; we say
$A \fsteq^\ebmu B$ iff $\supp A \fsteq^\ebmu \supp B$; we say
$A \fe^\ebmu B$ iff $\supp A \fe^\ebmu \supp B$.

\begin{pr}\label{pwitness}
Let $A,B \in \T$.  Then:
$A \fst B$ if and only if there exists $\bmu$ such that
$A \fst^\ebmu B$.
\end{pr}
\begin{proof}
The grid-based definitions must be used:
By Theorem~\ref{sgwt} there is $\ba$ such that
$A \fsteq^\ba \mag A$.  
And $\mag A \fst \mag B$,
so there is $\bb$
such that $\mag A \fst^\bb \mag B$.  Taking the union
$\bmu = \ba \cup \bb$, we get
$A \fst^\ebmu B$.
\end{proof}

\subsection*{Witnesses and Generators}

If $A \fst^\ebmu B$, we may say that $\bmu$ is a \Def{witness}
for $A \fst B$, or that $\bmu$ \Def{witnesses} $A \fst B$.
A given pair $A,B$ may of course have many
different witnesses.  If $\m,\n \in \G$ and $\m \fst\n$,
then it is witnessed by the singleton $\{\m/\n\}$.
Similarly, if $A \fsteq^\ebmu B$, we say $\bmu$ is
a witness for $A \fsteq B$; if $A \fe^\ebmu B$, we
say $\bmu$ is a witness for $A \fe B$.

For some purposes (such as computer algebra calculation) it may be
desirable to provide a witness for every assertion
$A \fst B$.  In \cite{edgar} we talked of keeping track
of generators, and providing addenda for the set
of generators.  Here, we will be doing this more
systematically.

Other ``witness'' terminology:  If $\AA \subseteq \G$ is a
subgrid, we say that $\ba$ is a \Def{witness} for $\AA$ if
$\AA \subseteq \g\ba^*$, where $\g = \max \AA$.
Theorem~\ref{sgwt} says that every subgrid has a witness.
(And of course this is the reason
we call it the Subgrid Witness Theorem.)
If $T \in \T$, then we say that $\ba$ is a
\Def{witness} for $T$ if $\ba$ is a witness for $\supp T$
in the sense just defined.
Thus: if $T \ne 0$,
then $\ba$ is a witness for $T$ if and only if
$\ba$ is a witness for $T \fsteq \mag T$.
That is, $\ba$ is a witness for $a\g(1+S)$
[where $a \in \R$, $a \ne 0$, $\g \in \G$, $S \in \T$, $S \fst 1$]
iff $S \fst^\ba 1$.  Given $\bmu$
with $\supp T \subseteq\GRID^{\bmu,\bm}$, to produce
a witness for $T$ we may need to augment $\bmu$ with
a smallness addendum for $S$.  Also note the extreme case:
if $\AA = \{\g\}$ is a singleton, then $\emptyset$ witnesses $\AA$.

For a subgrid $\AA \subseteq \G$
we will say $\bmu$ \Def{generates} $\AA$ iff $\AA \subseteq \GRID^{\ebmu,\bm}$
for some $\bm$.  And for a transseries $A$ we will say $\bmu$ \Def{generates}
$A$ iff $\bmu$ generates $\supp A$.  There are two conditions:
\begin{enumerate}
\item[({i})] $\bmu$ generates $A$
\item[({ii})] $\bmu$ witnesses $A$
\end{enumerate}
They are
related but not the same.  If $\bmu$ witnesses $A$, we may need
to add a generator for the monomial $\mag A$ to get a generator
for $A$.  On the other hand,
the usual example $1+xe^{-x}$ is generated by $\{x^{-1}, e^{-x}\}$
but not witnessed by it.  A witness can be obtained using a smallness
addendum $xe^{-x}$.

\begin{no}
$\T^\bb$ denotes the set of transseries generated by $\bb$;
$\TW\ba$ denotes the set of transseries witnessed by $\ba$;
$\TWG\ba\bb$ denotes the set of transseries witnessed by $\ba$
and generated by $\bb$.
\end{no}

\begin{re}
\textit{Closure properties.} (See Section~\ref{sec:begin}.)
The set $\T^\bb$ is closed under sums and products,
but in general not quotients.  The set $\TW\ba$ is closed under
products and quotients; but in general not sums.  The set $\TWG\ba\bb$
is closed under products, but in general not sums
or quotents.  The set $\TWG\ba\ba$ is closed under
products and quotents, but in general not sums.
\end{re}

\begin{ex}
If $A \sim B$ and $B \fst^\ebmu C$, it need not follow that
$A \fst^\ebmu C$.  For example:
$\ebmu = \{x^{-1},e^{-x}\}$, $A = x^{-1}+xe^{-x}$, $B = x^{-1}$, $C = 1$.
\end{ex}

\begin{pr}\label{wit_EL}
Let $A,B,C \in \T$ and let $\bmu$ be a ratio set.
If $A \sim B$, $B \fst^\ebmu C$ and $\bmu$ witnesses $A$,
then $A \fst^\ebmu C$.
\end{pr}
\begin{proof}
Let $\fa \in \supp A$.  Then $\fa \fsteq^\ebmu \mag A = \mag B \fst^\ebmu C$.
\end{proof}

\begin{ex}
If $A \fst^\ebmu B$ and $B \sim C$, it need not follow that
$A \fst^\ebmu C$.  For example:
$\bmu = \{x^{-1},e^{-x}\}$, $A = xe^{-x}$,
$B = 1+x^2e^{-x}$, $C = 1$.
\end{ex}

\begin{pr}\label{wit_LE}
Let $A,B,C \in \T$ and let $\bmu$ be a ratio set.
If $A \fst^\bmu B$, $B \sim C$, and
$\bmu$ witnesses $B$, then $A \fst^\ebmu C$.
\end{pr}
\begin{proof}
Let $\fa \in \supp A$.  Then there is $\fb \in \supp B$
with $\fa \fst^\ebmu \fb \fsteq^\ebmu \mag B = \mag C$.
\end{proof}

A natural partial order for ratio sets is inclusion of the generated
semigroups.  Let $\ba, \bb$ be ratio sets.  The following are equivalent:
\begin{enumerate}
\item[({i})] $\ba^* \supseteq \bb^*$
\item[({ii})] $\ba^+ \supseteq \bb^+$
\item[({iii})] $\ba^* \supseteq \bb$
\item[({iv})] For all $A,B \in \T$, if $A \fst^\bb B$, then $A \fst^\ba B$.
\end{enumerate}

\subsection*{Exponent Subgrids}
\begin{lem}[Support Lemma]\label{supportlemma}
If $U_1, \cdots, U_n \in \R\lbb \G \rbb$, then among the linear
combinations
$$
	\sum_{i=1}^n a_i U_i,\qquad a_i \in \Z
$$
there are only finitely many different magnitudes.
\end{lem}
\begin{proof}
There are at most $n$ different magnitudes among the
real linear combinations of $U_1,\cdots,U_n$.
Indeed, the set of real linear
combinations has dimension at most $n$.  If possible,
let $V_1, \cdots, V_{n+1}$ be linear combinations
of $U_1,\cdots,U_n$ with
$\mag V_1 > \mag V_2 > \dots > \mag V_{n+1}$.  Then,
since they are linearly dependent, there is some $k$
such that $V_k$ belongs to the linear span of
$\{V_{k+1},\cdots,V_{n+1}\}$.  But then
$\mag V_k \le \max\{\mag V_{k+1},\cdots,\mag V_{n+1}\}$,
a contradiction.
\end{proof}

\begin{lem}\label{expsubgrid}
Let $\AA \subseteq \G$ be a subgrid.  Let
$\AA_1 := \bigcup \supp L$ where the union is over
all $L$ such that $e^L \in \AA$.  Then $\AA_1 \subset \Glarge$
is also a subgrid.
\end{lem}
\begin{proof}
There is $\bmu = \{\mu_1,\cdots,\mu_n\}$ and $\bm \in \Z^n$
with $\AA \subseteq \GRID^{\ebmu,\bm}$.  Write
$\mu_i = e^{L_i}$, where $L_i \in \R\lbb\G\rbb$
is purely large.  Then for any $e^L \in \AA$,
the logarithm $L$ belongs to
$\SW := \SET{\sum_{i=1}^n p_i L_i}{\bp \in \Z^n}$.
So
$$
	\bigcup_{L \in \SW} \supp L \subseteq \bigcup_{i=1}^n \supp L_i
$$
is contained in a finite union of subgrids and is therefore a subgrid itself.
\end{proof}

\begin{de}
Call $\AA_1$ the \Def{exponent subgrid} of $\AA$.
\end{de}

There is a variant for use with log-free transseries and subgrids.

\begin{lem}\label{lfexpsubgrid}
Let $\AA \subseteq \G_\bullet$ be a subgrid.  Let
$\AA_1 := \bigcup \supp L$ where the union is over
all $L$ such that $x^be^L \in \AA$.  Then $\AA_1 \subset \Glarge_\bullet$
is also a subgrid.
\end{lem}
\begin{proof}
There is $\bmu = \{\mu_1,\cdots,\mu_n\}$ and $\bm \in \Z^n$
with $\AA \subseteq \GRID^{\ebmu,\bm}$.  Write
$\mu_i = x^{b_i}e^{L_i}$, where $b_i \in \R$ and $L_i \in \R\lbb\G_\bullet\rbb$
is purely large.  Proceed as before.
\end{proof}

\begin{re}
Let $\AA$ be a log-free subgrid.  If $\AA \subset \G_{N}$, $N \ge 1$,
then $\AA_1 \subset \G_{N-1}$.  If $\AA \subset \G_{0}$,
then $\AA_1 = \emptyset$.
\end{re}

\begin{de}
Call $\AA_1$ the \Def{log-free exponent subgrid} of $\AA$.
If $T \in \T_\bullet$, then the log-free exponent subgrid
of $\supp T$ is also called the log-free exponent subgrid
of $T$.  If $\bmu \subset \Gsmall_\bullet$ is a ratio set, it
is a finite set, so it is a subgrid.  So we will
sometimes refer to the log-free exponent subgrid of
a ratio set $\bmu$
(which is equal to the log-free exponent subgrid of
any grid $\GRID^{\ebmu,\bm}$).
\end{de}

\begin{de}
An \Def{exponent generator} for a subgrid $\AA \subset \G_\bullet$ is a
ratio set $\ba$ such that: $\ba$ is contained in the subgroup
generated by the log-free exponent subgrid of $\AA$
and $L \in \T^\ba$ for all $L$ with $x^be^L \in \AA$.  We say ``an''
exponent generator since there is more than one possibility.
Of course, if $\AA \subset \G_N$, then $\ba \subset \G_{N-1}$.
\end{de}

\subsection*{Heredity Addendum}
A ``heredity addendum'' is mentioned in \cite{edgar}.
Now we will discuss it more fully.

\begin{de}
Let $\BB \subseteq \G_\bullet$ be a log-free subgrid.
Then $\BB$ is \Def{hereditary} iff, for all
$x^b e^L \in \BB$ with $b \in \R$ and $L \in \T$ purely large
log-free, we have $\supp L \subseteq \BB$.
\end{de}

\begin{pr}\label{heredsubgrid}
Let $\AA \subseteq \G_\bullet$ be a log-free subgrid.
There is a hereditary log-free subgrid $\BB$ such that
$\BB \supseteq \AA$ and the height of $\BB$ is the
same as the height of $\AA$.
\end{pr}
\begin{proof}
The proof is by induction on the height.
Suppose first that $\AA$ has height $0$, so
$\AA \subseteq \G_0 = \SET{x^b}{b \in \R}$.
Take $\BB = \AA$.
If $x^be^L \in \AA$, then $L = 0$, so $\supp L \subseteq \AA$
vacuously.

Now suppose $\AA \subseteq \G_N$, $N > 0$, and the result is
known for height $N-1$.  Let $\AA_1$ be the
log-free exponent subgrid of $\AA$.  So $\AA_1 \subseteq \G_{N-1}$,
and there is a hereditarty log-free subgrid
$\BB_1 \subseteq \G_{N-1}$ with $\BB_1 \supseteq \AA_1$.
Let $\BB = \AA \cup \BB_1$.
Then $\BB \supseteq \AA$ is a log-free subgrid of height $N$.
I must show $\BB$ is hereditary.
Let $x^b e^L \in \BB$.  If $x^be^L \in \AA$,
then $\supp L \subseteq \AA_1 \subseteq \BB$.
If $x^be^L \in \BB_1$, then
$\supp L \subseteq \BB_1 \subseteq \BB$.
So $\BB$ is hereditary.  This completes the induction.
\end{proof}

\begin{re}
Let $\AA$ and $\BB$ be hereditary log-free subgrids.
Then $\AA \cup \BB$ and $\AA\cdot\BB \cup \AA \cup \BB$ are also
hereditary log-free subgrids.
\end{re}

\begin{re}
Let $\bmu = \{x^{b_1}e^{L_1},\cdots,x^{b_n}e^{L_n}\}$
be a log-free ratio set.  Then $\GRID^{\ebmu,\bm}$
is hereditary iff $\supp L_i \subseteq \GRID^{\ebmu,\bm}$
for $1 \le i \le n$.  If
$$
	\bigcup_{i=1}^n \supp L_i \subseteq \GRID^\ebmu ,
$$
then $\GRID^{\ebmu,\bm}$ is hereditary for some $\bm$, and in that
case we may abuse the above terminology and
say simply that $\bmu$ is hereditary.
\end{re}

\begin{pr}\label{derivgen}
Let $\bmu$ be a hereditary log-free ratio set.
Let $T \in \T$.  If $\supp T \subseteq \GRID^\ebmu$, then
$\supp(xT') \subseteq \GRID^\ebmu$ and
$\supp\big((xT)'\big) \subseteq \GRID^\ebmu$.  Assume also that
$x^{-1} \in \GRID^\ebmu$.
If $\supp T \subseteq \GRID^\ebmu$,
then $\supp(T') \subseteq \GRID^\ebmu$.
\end{pr}
\begin{proof}
We first consider $xT'$.  This is proved by induction on
the height.  First consider height $0$.  If
$\g \in \GRID^\ebmu$, $\g \in \G_0$, say $\g = x^b$, then
$\g' = bx^{b-1}$ and $x\g' = b\g$, so
$\supp(x\g') \subseteq \GRID^\ebmu$.  If
$\supp T \subseteq \GRID^\ebmu \cap \G_0$,
then
$$
	\supp(xT') = x\supp(T') \subseteq
	x\left(\bigcup_{\g \in \supp T} \supp(\g')\right)
	=\left(\bigcup_{\g \in \supp T} \supp(x\g')\right)
	\subseteq \GRID^\ebmu .
$$
Assume it is true for height $N-1$.  If $\g \in \GRID^\ebmu$,
$\g \in \G_N$, say $\g = x^be^L$, then
$\g' = (bx^{-1}+L')\g$ and $x\g' = (b+xL')\g$.  By
the induction hypothesis, $\supp(xL') \subseteq \GRID^\ebmu$.
Since $\GRID^\ebmu$ is closed under multiplication, we have
$\supp(x\g') \subseteq \GRID^\ebmu$.  If
If $\supp T \subseteq \GRID^\ebmu \cap \G_N$, then
add as before.

Next consider $(xT)'$.  We have $(xT)' = T + xT'$, and
both terms have support in $\GRID^\ebmu$, so also
$\supp((xT)') \subseteq \GRID^\ebmu$.

In case $x^{-1} \in \GRID^\ebmu$, when we have
$\supp(xT') \subseteq \GRID^\ebmu$ we will also get
$\supp(T') \subseteq \GRID^\ebmu$.
\end{proof}

\section{Beginning Witnesses}\label{sec:begin}
We begin with the basic things to be checked concerning the ratio sets.
Some of them were already spelled out in \cite{edgar}.

\begin{pr}[{\cite[\Esubgrid]{edgar}}]
If $\AA, \BB$ are subgrids, then so are
$\AA \cup \BB$ and $\AA\cdot \BB$.  Thus: if $S,T \in \T$, then so are
$S+T$ and $ST$.
\end{pr}

\begin{pr}\label{sumproduct}
If $\bmu$ generates both $S$ and $T$,
then $\bmu$ generates $S+T$ and $ST$.
\end{pr}

\begin{pr}\label{witproduct}
If $\bmu$ witnesses both $S$ and $T$,
then $\bmu$ also witnesses $ST$.
\end{pr}

\begin{re}
But possibly not $S+T$:
For example, $S = x+1$, $T = -x+xe^{-x}$, $\bmu = \{x^{-1},e^{-x}\}$.
\end{re}

\begin{pr}\label{witless1}
If $\bmu$ witnesses both $S \fst 1$ and $T \fst 1$, then
$\bmu$ witnesses $ST \fst 1$ and $S+T \fst 1$.
\end{pr}

\subsection*{Multiply Far-Greater Relations}

It was noted in \cite{edgar} that
$A \fst^\ebmu B$ need not imply $A S \fst^\ebmu B S$,
even if $\bmu$ generates $A, B, S$.  The ``witness''
concept can overcome this.

\begin{pr}\label{productwitness}
Let $A, B, S \in \T$.  Assume $\bmu$ witnesses
either $B$ or $S$.
If $A \fst^\ebmu B$, then $A S \fst^\ebmu B S$.
If $A \fsteq^\ebmu B$, then $A S \fsteq^\ebmu B S$.
\end{pr}
\begin{proof}
Let $\m \in \supp(AS)$.  Then there exist $\fa_0 \in \supp A$
and $\g_0 \in \supp S$ with $\m = \fa_0\g_0$.  There is
$\fb_0 \in \supp B$ with $\fa_0 \fst^\ebmu \fb_0$.  Let
\begin{align*}
	\fb_1 &= \max\SET{\fb \in \supp B}{\fb \fgteq^\ebmu \fb_0},
	\\
	\g_1 &= \max\SET{\g \in \supp S}{\g \fgteq^\ebmu \g_0},
\end{align*}
which exist because these supports are well ordered.
Now we have assumed that $\bmu$ witnesses either $B$ or $S$.
The two cases are similar, so assume $\bmu$ witnesses $S$.
Then $\g_1 = \mag S$.
Let $\n = \fb_1\g_1$.  I claim $\n \in \supp(BS)$.  Assume not:
it must be because of cancellation in the product $BS$.  So
there exist $\fb_2 \in \supp B$ and $\g_2 \in \supp S$
so that $\fb_1\g_1 = \fb_2\g_2$ but $\fb_1 \ne \fb_2$
and $\g_1 \ne \g_2$.  Now $\g_1 = \mag S$ and $\bmu$ witnesses $S$,
so $\g_2 \fst^\ebmu \g_1$.  That means $\g_2/\g_1 \in \bmu^+$.
But $\fb_1/\fb_2 = \g_2/\g_1$, so $\fb_1 \fst^\ebmu \fb_2$,
which contradicts the maximiality of $\fb_1$.  This contradiction
shows that $\n \in \supp(BS)$.  Now
\begin{equation*}
	\m = \fa_0 \g_0 \fst^\ebmu \fb_0 \g_0 \fsteq^\ebmu
	\fb_1 \g_1 = \n,
\end{equation*}
so $\m \fst^\ebmu \n$.  Therefore $AS \fst^\ebmu BS$.

The second assertion is proved similarly.
\end{proof}

\begin{ex}
False in general: $A_1 \fst^\ebmu A_2, B_1 \fst^\ebmu B_2
\Longrightarrow A_1B_1 \fst^\ebmu A_2B_2$.  (It is
true for monomials.)  Take
$\bmu = \{x^{-1},e^{-x}\}$.  Then
\begin{align*}
	&x^{-3} \fst^\ebmu x^{-2}+e^{-2x},
	\qquad\text{and}\qquad
	e^{-3x} \fst^\ebmu x^{-2}-e^{-2x},
	\\ \text{but not }\quad &
	x^{-3}e^{-3x} \fst^\ebmu \big(x^{-2}+e^{-2x}\big)
	\big(x^{-2}-e^{-2x}\big) = x^{-4}-e^{-4x} .
\end{align*}
\end{ex}

\begin{pr}\label{prodwitnessmult}
Let $A_1, A_2, B_1, B_2 \in \T$, let $\bmu$ be a ratio set.
Assume $A_1 \fst^\ebmu A_2$, $B_1 \fst^\ebmu B_2$, and
$\bmu$ witnesses $B_2$.  Then $A_1B_1 \fst^\ebmu A_2B_2$.
\end{pr}
\begin{proof}
Apply Proposition~\ref{productwitness} twice:
$A_1B_1 \fst^\ebmu A_1B_2$ and $A_1B_2 \fst^\ebmu A_2B_2$.
\end{proof}

\subsection*{Laurent Series}
If $\ba$ witnesses $S$ and $S \fst^\ba 1$, then
$\ba$ witnesses (and generates)
the sum $A = \sum_{j=p}^\infty a_j S^j$.  If $p \ge 1$, then
$\ba$ witnesses $A \fst 1$.

Let $S_1 \fst 1, \cdots, S_m \fst 1$.  Assume $\ba_i$ witnesses
$S_i \fst 1$ for $1 \le i \le m$.  Consider the sum
$$
	A = \sum_{j_1=p_1}^\infty\,\sum_{j_2=p_2}^\infty\cdots
	\sum_{j_m=p_m}^\infty\;c_{j_1 j_2 \dots j_m}
	S_1^{j_1} S_2^{j_2} \cdots S_m^{j_m} .
$$
If the ``leading coefficient'' $c_{p_1 p_2\dots p_m}$ is not zero,
then $\bb := \bigcup_{j=1}^m \ba_j$ witnesses $A$.  But as with a finite sum,
an addendum may be required in general.

\begin{pr}\label{D_power}
Let $A \ne 0$ and assume $\ba$ witnesses $A$.
Then $\ba$ witnesses $A^{-1}$.  And
$\ba$ witnesses $A^b$ for any $b \in \R$.
\end{pr}
\begin{proof}
Write $A = a e^L (1+S)$ [with $a \in \R$, $a\ne 0$,
$L$ purely large, $S$ small] so $\ba$ witnesses
$S \fst 1$, and
$$
	A^b = a^b e^{bL} \sum_{j=0}^\infty \binom{b}{j} S^j 
$$
so $\ba$ witnesses $A^b$.
\end{proof}

\begin{re}
A generator for $A^b$ is
$\ba \cup \{e^{\pm b L}\}$, the sign chosen so that the monomial
is small.
If $bL<0$, then $\ba \cup \{e^{b L}\}$ witnesses $A^b \fst 1$.
\end{re}

\begin{re}
Let $\ba$ be a ratio set.
Then $\SET{A \in \T}{A \ne 0, \ba \text{ witnesses } A}$ is closed under
products and quotients.
\end{re}

\subsection*{Logarithm and Exponential}

\begin{pr}\label{D_log}
Let $A = a e^L(1+S)$, where $a \in \R$, $a > 0$,
$L$ is purely large, $S$ is small.  If $\ba$ witnesses $L$ and
$\bb$ witnesses $S \fst 1$, then
$$
	\bmu := \ba \cup \bb \cup \{(\mag L)^{-1}\}
	\quad\text{witnesses}\quad
	\log A = L + \log a - \sum_{j=1}^\infty \frac{(-1)^j S^j}{j} .
$$
\end{pr}

Also: $\bmu$ generates $\log A$; $\bb$ witnesses $\sm(\log A) \fst 1$;
if $A \sim 1$, then $\bb$ witnesses $\log A \fst 1$;
if $A \fe 1$, then $\bb$ witnesses and generates $\log A$.

Note: If $A \not\fe 1$, then $\log A \fgt 1$.

\begin{pr}\label{D_exp}
Let $A = L + c + S$, where $L=\la A$, $c=\const A$,
and $S = \sm A$.  If $\ba$ witnesses $S \fst 1$, then
$\ba$ witnesses
$$
	e^A = e^c e^L \sum_{j=0}^\infty \frac{S^j}{j!}
$$
and $\bmu := \ba \cup \{e^{\pm L}\}$ generates $e^A$.

If $L<0$ {\rm(}that is, $A$ is large and negative{\rm)}, then
$\bmu := \ba \cup \{e^L\}$ witnesses $e^A \fst 1$.
\end{pr}

\subsection*{Series}
If $S = \sum A_i$ is $\bmu$-convergent, then of course
there is a witness for $S$.  But is there a single
witness for all the terms $A_i$?
In general, there is no such witness.

\begin{ex}\label{nongeom}
Let $\bmu = \{x^{-1},e^{-x}\}$ and for $j \in \N$ let
$A_j = x^{-2j} + x^{-j-1}e^{-x}$:
\begin{align*}
	A_1 &= x^{-2}+x^{-2}e^{-x}, \\
	A_2 &= x^{-4}+x^{-3}e^{-x}, \\
	A_3 &= x^{-6}+x^{-4}e^{-x}, \\
	A_4 &= x^{-8}+x^{-5}e^{-x}, \\
	A_5 &= x^{-10}+x^{-6}e^{-x}, \cdots
\end{align*}
Of course $S = \sum A_j$ is $\bmu$-convergent, since in that sum
each monomial occurs at most once.  And $\bmu$ witnesses $S$.  Now
$A_j = x^{-2j}\big(1+x^{j-1}e^{-x}\big)$, so if
$\ba$ witnesses $A_j$, then $x^{j-1}e^{-x} \fst^\ba 1$.
But since the set $\SET{x^{j-1}e^{-x}}{j \in \N}$ is not
well-ordered, it is not contained in any grid, and in particular
it is not contained in $\ba^+$.
\end{ex}

\subsection*{Geometric Convergence}
There is a ``more rapid'' type of convergence for series (and sequences).
Compare it to ``pseudo convergence'' commonly used in valuation theory
\cite{gravett}.
The terms of the series decrease at a rate specified
by a ratio set $\bmu$.  [The ``ratio'' in the name comes
from this usage: the ratio of consecutive terms in a series.]

\begin{de}\label{geom_conv}
Let $\bmu$ be a ratio set.  Let $A_j \in \T$ for $j \in \N$.
The series $\sum_{j=1}^\infty A_j$ is said to be
\Def{$\bmu$-geometrically convergent}
if $\bmu$ witnesses $A_j$ and $A_{j} \fgt^\ebmu A_{j+1}$ for
all $j$.

A series is said to be \Def{geometrically convergent}
if it is $\bmu$-geometrically convergent for some $\bmu$.
\end{de}

Example~\ref{nongeom} is convergent but not geometrically convergent.

\begin{pr}\label{geom_basic}
Assume $\sum A_j$ is $\bmu$-geometrically convergent.
Then: All $A_j$ are supported by the subgrid $(\mag A_1)\bmu^*$.
The series $\sum A_j$ converges in the point-finite sense.
The sum $S = \sum A_j$ is witnessed by $\bmu$ and $S \sim A_1$.
\end{pr}

\begin{de}
A sequence $S_j, j = 1,2,3,\cdots$ is said to be
\Def{$\bmu$-geometrically Cauchy} if $\bmu$ witnesses $S_{j+1}-S_j$
and $S_{j+1}-S_j \fgt^\ebmu S_j - S_{j-1}$ for all $j$.
(Compare this to the usual ``pseudo Cauchy'' \cite{gravett}.)
\end{de}

This means the series $\sum_{j=1}^\infty (S_{j+1}-S_j)$
is $\bmu$-geometrically convergent in the sense above.  And of course
$S_j$ converges in the asymptotic (Costin) topology.

\begin{de}
Let $S_j, S \in \T$.  We say the sequence $S_j$ is $\bmu$-geometrically
convergent to $S$ if
$\bmu$ witnesses $S-S_j$ and $S-S_j \fgt^\ebmu S-S_{j+1}$
for all $j$.  (It follows that $S_j \to S$.
Of course $S_j - S \sim S_j - S_{j+1}$
follows, so this is also pseudo convergence.)
\end{de}

\begin{pr}\label{geometric}
Let $S_j$ be $\bmu$-geometrically Cauchy.  Then there is $S$
so that $S_j$ converges $\bmu$-geometrically to $S$.
\end{pr}
\begin{proof}
Let $S_n = \sum_{j=1}^n A_j$, so that $T = \sum_{j=2}^\infty A_j$
is $\bmu$-geometrically convergent.  So
$S_n$ converges to $S = S_1 + T$.  Now
$S-S_n = \sum_{j=n+1}^\infty A_j$, which is $\bmu$-geometrically
convergent, so $\bmu$ witnesses $S-S_n$ and
$S-S_n \sim A_{n+1}$.  Also
$S-S_n \sim A_{n+1} \fgt^\ebmu S-S_{n+1}$, so
$S-S_n \fgt^\ebmu S-S_{n+1}$ by Proposition~\ref{wit_LE}.
\end{proof}

\begin{re}
The usual version of this in valuation theory would be:
the series $\sum A_j$ is \Def{pseudo Cauchy}
iff $A_j \fgt A_{j+1}$ for all $j$.  The sequence
$S_n$ is \Def{pseudo Cauchy} iff
$S_j-S_{j-1} \fgt S_{j+1}-S_j$ for all $j$.
(This is often also used for sequences indexed by ordinals.)
The sequence $S_j$ is \Def{pseudo convergent} to $S$ iff
$S-S_j \sim S_{j+1}-S_j$ for all $j$.  This will be the useful
notion only for well based transseries spaces.
For example, $\sum_{j=1}^\infty x^{-\log n}$ is pseudo Cauchy,
but its sum is not grid based.
Also: pseudo convergence does not imply convergence (in any
of the three senses of \cite[\Cconverge]{edgarc}).  For example
$S_j = x^{-j}e^x+x^je^{-x}$ is pseudo convergent to $0$.
Also, in the well-based case, where $ \T \subsetneq \R[[\G]]$,
there exist pseudo Cauchy sequences in $\T$ with pseudo limits
only in $\R[[\G]] \setminus \T$.
\end{re}

\begin{lem}[Summation Lemma]\label{summationlemma}
Let $\bmu \subset \Gsmall$ be a ratio set.
Assume $\bmu$ witnesses $V$, the series $S = \sum B_j$
converges $\bmu$-geometrically, $\bmu$ witnesses
$A_j$, and $A_j \sim B_j V$ for $j = 1,2,3,\cdots$.
Then $T = \sum A_j$ converges $\bmu$-geometrically
and $T \sim SV$.
\end{lem}
\begin{proof}
By definition $\bmu$ witnesses $B_j$ and
$B_j \fgt^\ebmu B_{j+1}$ for all $j$.
By Proposition~\ref{witproduct} $\bmu$ witnesses $B_jV$.
By Proposition~\ref{productwitness} $\bmu$ witnesses
$B_jV \fgt^\ebmu B_{j+1}V$.  So
$A_j \sim B_j V \fgt^\ebmu B_{j+1} V \sim A_{j+1}$,
and by Propositions \ref{wit_EL} and~\ref{wit_LE},
$A_j \fgt^\ebmu A_{j+1}$.  So $T = \sum A_j$
converges $\bmu$-geometrically.  Finally,
$\mag T = \mag A_1 = \mag(B_1V) = \mag B_1 \mag V = \mag S \mag V = \mag(SV)$
so $T \sim SV$.
\end{proof}

\subsection*{Geometric Convergence of Multiple Series}
Geometric convergence of series adapts well to multiple series.

\begin{de}
Let $n \ge 2$ be an integer.  An \Def{$n$-fold multiple series}
is a series indexed by $\N^n$:
$$
	\sum_{\bp \in \N^n} A_\bp .
$$
Let $\bmu$ be a ratio set.  We say the $n$-fold multiple series
$\sum A_\bp$ is \Def{$\bmu$-geometrically convergent} iff:
$\bmu$ witnesses $A_\bp$ for all $\bp \in \N^n$,
$A_\0 \ne 0$, and
for all $\bp,\bq \in \N^n$, if $\bp < \bq$, $A_\bp \ne 0$,
and $A_\bq \ne 0$, then $A_\bp \fgt^\ebmu A_\bq$.
\end{de}

\begin{re}\label{gridmultiseries}
A grid-based transseries is, of course, the primary example of this.
Let $T \in \TW\ebmu$.
Write $\bmu = \{\mu_1,\cdots,\mu_n\}$
with $\mu_1 \fgt \cdots \fgt \mu_n$.
Then $\supp T \subseteq \m\bmu^* = \SET{\m\bmu^\bp}{\bp \in \N^n}$
where $\m = \mag T$.  And the ``formal'' series
$$
	T = \sum_\g T[\g]\g = \sum_{\bp \in \N^n} a_\bp \cdot \m\bmu^\bp
$$
is a $\bmu$-geometrically convergent $n$-fold multiple series.
(If the representation of $\g$ as $\m\bmu^\bp$ is unique,
then the coefficient $a_\bp$ must be $T[\m\bmu^\bp]$.  But if
it is not unique, then there is more than one choice
for the coefficients.)
\end{re}

\begin{pr}
Assume $\sum A_\bp$ is $\bmu$-geometrically convergent.  Then all
$A_\bp$ are supported by the subgrid $(\mag A_\0)\bmu^*$.
The series $\sum A_\bp$ converges in the point-finite sense.
The sum $S = \sum A_\bp$ is witnessed by $\bmu$ and $S \sim A_\0$.
\end{pr}

\begin{lem}[Multiple Summation Lemma]\label{multsummationlemma}
Let $\bmu \subset \Gsmall$ be a ratio set.
Assume $\bmu$ witnesses $V$, the series $S = \sum B_\bp$
converges $\bmu$-geometrically, $\bmu$ witnesses
$A_\bp$, and $A_\bp \sim B_\bp V$ for all $\bp \in \N^n$.
Then $T = \sum A_\bp$ converges $\bmu$-geometrically
and $T \sim SV$.
\end{lem}

The proof of Lemma~\ref{summationlemma} adapts with no difficulty.

\section{Derivative}
\cite[\Ederivpropertiesa]{edgar} states
$\m \fst \n \Longrightarrow \m' \fst \n'$ for monomials $\m,\n$.
Here is the ``witness'' version.

\begin{pr}\label{monoderivcompare}
Let $\bmu \subset \Gsmall$ be a ratio set.
Then there is a ratio set $\ba$ such that:
{\rm(a)}~$\ba^* \supseteq \bmu$;
{\rm(b)}~if $\m \in \GRID^\ebmu$, then $\m' \in \TWG\ba\ba$;
{\rm(c)}~for all $\m,\n \in \GRID^\ebmu$, if $\m \fst^\ebmu \n$
and $\n \ne 1$, then $\m' \fst^\ba \mag(\n')$,
so that $\m' \fst^\ba \n'$.
\end{pr}
\begin{proof}
(I) We begin with the case where $\bmu \subseteq \Gsmall_{N,-1}$, $N \ge 1$.
That is, every monomial $\mu_i \in \bmu$ has the form
$e^{L_i}$ with $L_i \in \R\lbb\G_{N-1}\rbb$ purely large and log-free.
Order $\bmu = \{\mu_1,\cdots,\mu_n\}$ as usual so that
$1 \fgt \mu_1 \fgt \mu_2 \fgt \cdots \fgt \mu_n$.
So $0 > L_1 > L_2 > \dots > L_n$ and thus
$L_1 \fsteq L_2 \fsteq \cdots \fsteq L_n$
and $L_1' \fsteq L_2' \fsteq \cdots \fsteq L_n'$
[the $L_i$ are large, so do not have magnitude $1$].

Let $1 \le j \le n$, $P = \sum_{i=1}^n p_i L_i'$, and
$Q = \sum_{i=1}^n q_i L_i'$,
$p_i, q_i \in \Z$.  If $Q \fsteq L_j'$ and $P \fsteq L_j'$, then
$Q\mu_j \fst P$ by
``Height Wins'' \cite[\Eheightwins]{edgar}, since
$\mu_j = e^{L_j}$ has greater height than both $L_j$ and $L_j'$.

Write
$$
	\SW := \SET{\sum_{i=1}^n p_i L_i'}{\bp \in \Z^n}.
$$
By the Support Lemma \ref{supportlemma},
$\SET{\mag(Q)}{Q \in \SW}$ is a finite set of monomials.
So we may define $\ba$ so that
$\ba^* \supseteq \bmu$ and:
\begin{enumerate}
\item[(i)]
$\ba$ generates $\mag Q$ for all $Q \in \SW$

\item[(ii)]
$\ba$ witnesses $Q$ for all $Q \in \SW$
[by (i) and (ii), $\ba$ generates all $Q \in \SW$]

\item[(iii)]
$\ba$ witnesses $\mag(P) \fgt \mag(Q)\mu_j$
for all $j$, $1 \le j \le n$, and all $P,Q \in \SW$
such that $Q\fsteq L_j'$ and $P\fsteq L_j'$.
\end{enumerate}

\textit{Claim: if $1 \le j \le n$, $P,Q \in \SW$, and $Q-P \fsteq L_j'$,
then $\mag(P) \fgt^\ba \mag(Q)\mu_j$.}  Indeed, in case $Q \fsteq L_j'$
it follows that $P \fsteq L_j'$ also and the claim follows from (iii).
In the other case $Q \fgt L_j'$ it follows that
$P \fe Q$ so that $\mag(P) \fgt^\ebmu \mag(P)\mu_j = \mag(Q)\mu_j$.

(a) holds by construction.

(b) Let $\m \in \GRID^\ebmu$.  Then the derivative is
$$
	\m' = \left(\sum_{i=1}^n p_i L_i'\right)\m = \m^\dagger\m
$$
[We used notation $\m^\dagger = \m'/\m$ for the logarithmic derivative of $\m$.]
Now $\m^\dagger \in \SW$ so, as noted, $\ba$ generates and
witnesses $\m^\dagger$.
Thus $\ba$ generates and witnesses $\m' = \m^\dagger \m$.

(c) Now let $\m,\n \in \GRID^\ebmu$ with $\m \fst^\ebmu \n$
and $\n \ne 1$.  Say $\m = \bmu^\bp$, $\n = \bmu^\bq$,
with $\bp > \bq$ in $\Z^n$.  The derivatives are:
$$
	\m' = \left(\sum_{i=1}^n p_i L_i'\right)\m = \m^\dagger\m ,\qquad
	\n' = \left(\sum_{i=1}^n q_i L_i'\right)\n = \n^\dagger\n .
$$
Let $j$ be largest such that $p_j \ne q_j$.
Then $\m^\dagger - \n^\dagger$ is a linear combination of
$L_1',\cdots, L_j'$, and thus $\m^\dagger - \n^\dagger \fsteq L_j'$.
So $\mag(\n^\dagger) \fgt^\ba \mag(\m^\dagger)\mu_j$.
If $\g \in \supp(\m')$, then $\g = \g_1\m$ where $\g_1 \in \supp(\m^\dagger)$;
and $\m^\dagger \in \SW$, so $\g_1 \fsteq^\alpha \mag(\m^\dagger)$
by (ii).  Also $\m/\mu_j \fsteq^\ebmu \n$ since $p_j > q_j$.  Thus:
\begin{equation*}
	\g = \g_1\m \fsteq^\ba \mag(\m^\dagger)\m =
	\big(\mag(\m^\dagger)\mu_j\big)\big(\m/\mu_j\big) \fst^\ba \mag(\n^\dagger)\n
	=\mag(\n') .
\end{equation*}
This shows $\m' \fst^\ba \mag(\n')$ and thus that
$\m' \fst^\ba \n'$.

(II) Now let $\bmu$ be any ratio set.  Say $\bmu \subset \G_{N,M-1}$,
$N \ge 1, M \ge 1$. Since $\G_{n,m} \subseteq \G_{n+1,m+1}$ where we
identify $\g\circ\log_m \in \G_{n,m}$ with
$(\g \circ \exp) \circ \log_{m+1} \in \G_{n+1,m+1}$, this includes the
general case.  Given such $\bmu$, define
$$
	\tbmu  := \SET{\g\circ\exp_M}{\g \in \bmu} ,
$$
so that $\tbmu \subset \G_{N,-1}$.  Construct the corresponding
ratio set $\widetilde\ba$ from $\tbmu$ as in~(I).
Then define
$\ba := \SET{\widetilde{\g}\circ\log_M}{\widetilde{\g} \in
\widetilde{\ba}} \cup \{\l_M'\}$.  Recall that
$\l_M'$ is a small monomial, since it is a finite product
of the form $(x \log x \log_2 x \cdots)^{-1}$.

(a) Of course $\ba^* \supseteq \bmu$ since
$\widetilde{\ba}^* \supseteq \tbmu$.

(b) Let $\m \in \GRID^\ebmu$.  Then $\m = \widetilde{\m}\circ\log_M$
where $\widetilde{\m} \in \GRID^\tebmu$.  So by (I)
$\widetilde{\ba}$ generates and witnesses $\widetilde{\m}'$, and therefore
$\ba$ generates and witnesses $\widetilde{\m}'\circ \log_M$.  But
$$
	\m' = (\widetilde{\m}\circ \log_M)'
	= \big(\widetilde{\m}'\circ\log_M\big)\cdot\l_M'
$$
and $\log_M' \in \ba$, so $\ba$ generates and witnesses $\m'$.

(c) Now let $\m,\n \in \GRID^\ebmu$ with $\m \fst^\ebmu \n$
and $\n \ne 1$.  Then $\m = \widetilde{\m} \circ \log_M$,
$\n = \widetilde{\n} \circ \log_M$, where
$\widetilde{\m}, \widetilde{\n} \in \GRID^\tebmu$ with
$\widetilde{\m} \fst^\tebmu \widetilde{\n}$,
$\widetilde{\n} \ne 1$.  So by (I) we have
$\widetilde{\m}' \fst^{\tilde{\ba}} \widetilde{\n}'$.
Therefore
\begin{align*}
	\m' &= (\widetilde{\m}\circ\log_M)'
	= \big(\widetilde{\m}'\circ\log_M\big)\cdot\l_M'
	\\ &\fst^{\ba} \big(\widetilde{\n}'\circ\log_M\big)\cdot\l_M'
	= (\widetilde{\n}\circ\log_M)'
	= \n' ,
\end{align*}
as required.
\end{proof}

\begin{de}
We will say that $\ba$ is a \Def{derivative addendum} for $\bmu$.
\end{de}

\begin{ex}\label{ex:derivativeaddendum}
Computations from this proof:

$\bmu = \{e^{-a_1x},\cdots, e^{-a_nx}\} \subset \Gsmall_{0,-1}$
leads to $\ba = \bmu$.

$\bmu = \{x^{-a_1},\cdots,x^{-a_n}\} \subset \Gsmall_0$ leads to
$\ba = \bmu \cup \{x^{-1}\}$.

$\bmu = \{e^{-x},e^{-e^x}\}$ leads to $\ba = \{e^{-x},e^x e^{-e^x}\}$.

$\bmu = \{x^{-1},e^{-x}\}$ leads to $\ba = \{x^{-1},xe^{-x}\}$.
\end{ex}

\begin{ex}
Of course Proposition~\ref{monoderivcompare}(c) does not say:
$$
	\textit{For all $\m,\n \in \G$, if $\m \fst^\ebmu \n$ and $\n \ne 1$,
	then $\m' \fst^\ba \n'$.}
$$For example, if $\bmu = \{x^{-1}\}$,
then there is no finite ratio set $\ba$ such that
$(x^{-1}\n)' \fst^\ba \n'$ for all $\n \in \G$.  We can see
this by considering $\n = \exp_k$ for $k=1,2,3,\cdots$.
\end{ex}

\begin{pr}
If $\bmu \subset \G_{N,M}$, then the derivative addendum
$\ba$ may be chosen so that $\ba \subset \G_{N,M+1}$.
\end{pr}
\begin{proof}
Examine the proof to see first: if $\bmu \subset \G_{N,-1}$,
then $\ba \subset \G_{N}$.
\end{proof}

\begin{re}\label{gridderiv}
Consider a grid $\GRID^{\ebmu,\bm}$.  In the preceding proposition,
if $\n \in \GRID^{\ebmu,\bm}$, then
$\supp \n' \subseteq \GRID^{\ba,\widetilde{\bm}}$,
where $\widetilde{\bm}$ is chosen so that
$\mag((\bmu^\bm)') = \ba^{\widetilde{\bm}}$.  This works as long
as $\bm \ne \0$.  Now consider the grid $\GRID^{\ebmu,\0}$.
Of course $\GRID^{\ebmu,\0} \subseteq \GRID^{\ebmu,\bm}$, where
$\bm = (-1,0,\cdots,0)$.
So choose $\widetilde{\bm}$
where $\mag((\mu_1^{-1})') = \ba^{\widetilde{\bm}}$.
[Recall that $\ba$ witnesses $(\mu_1^{-1})'$.]
\end{re}

\begin{pr}\label{derivconverge}
Let $\bmu$ be a ratio set, and let $\ba$ be a derivative
addendum for $\bmu$ as in
Proposition~\ref{monoderivcompare}.
Let $\sum_{i \in I} T_i$ be $\bmu$-convergent.
Then $\sum T_i'$ is $\ba$-convergent.
\end{pr}
\begin{proof}
There is a grid $\GRID^{\ebmu,\bm}$ that supports all $T_i$,
so by Remark~\ref{gridderiv} there is a grid $\GRID^{\ba,\widetilde{\bm}}$
that supports all $T_i'$. So it remains to show that the series
$\sum T_i'$ is point-finite.  Suppose, to the contrary, that
there is $\g$ such that $\AA = \SET{i \in I}{\g \in \supp(T_i')}$ is infinite.
For $i \in \AA$ there is $\n \in \supp(T_i)$ with $\g \in \supp(\n')$.
Since $\sum T_i$ is point-finite, there are infinitely many
different $\n \in \bigcup\supp(T_i)$ with $\g \in \supp(\n')$.
This is contained in a grid $\GRID^{\ebmu,\bm}$, so
there is an infinite sequence $\n_1 \fgt^\bmu \n_2 \fgt^\ebmu \cdots$
of such monomials.  (Of course $1$ is not in this sequence.)
But then by Proposition~\ref{monoderivcompare},
$\n_1' \fgt^\ebmu \n_2' \fgt^\ebmu \cdots$.  So the sequence
$\supp(\n_1'),\supp(\n_2'),\cdots$
is point-finite by \cite[\Emudominateprop]{edgar}.
So in fact $\g$ cannot belong to all of them.
This contradiction completes the proof.
\end{proof}

\begin{pr}\label{derivcompare}
Let $\bmu$ be a ratio set, and let $\ba$ be a derivative addendum
for $\bmu$ as
defined in Proposition~\ref{monoderivcompare}.  For all
$S, T \in \T^\ebmu$, if $S \fst^\ebmu T$,
$T \not\fe 1$, and $\bmu$ witnesses $T$,
then $S' \fst^\ba T'$.
\end{pr}
\begin{proof}
Let $\m \in \supp(S')$.  Then there is $\fa \in \supp S$
with $\m \in \supp(\fa')$.  There is $\fb \in \supp T$
with $\fa \fst^\ebmu \fb$.  Since $\bmu$ witnesses $T$,
$\fb \fsteq^\ebmu \mag T$.  So $\fa \fst^\ebmu \mag T$.
Then $\fa' \fst^\ba (\mag T)'$
by Proposition~\ref{monoderivcompare}(c).  There is
$\n \in \supp((\mag T)')$ with $\m \fst^\ba \n$.
But $\mag T \in \GRID^\ebmu$, so
$\ba$ witnesses $(\mag T)'$ by Proposition~\ref{monoderivcompare}(b).
Thus $\n \fsteq^\ba \mag((\mag T)')$ and therefore
$\m \fst^\ba \mag((\mag T)')= \mag(T') \in \supp(T')$.
This shows $S' \fst^\ba T'$.
\end{proof}

\begin{ex}
The hypothesis ``$\bmu$ witnesses $T$'' cannot be omitted
in Proposition~\ref{derivcompare}.  Let $\bmu = \{x^{-1},e^{-x}\}$.
Consider $S = x^{-1}$ and $T = x^{-j}e^x + 1$ for any $j \in \N$.
We have $\bmu$ witnesses and generates $S$, $\bmu$ generates $T$,
but $\bmu$ does not witness $T$.  Of course $S \fst^\ebmu T$
since $x^{-1} \fst^\ebmu 1$.  Compute
\begin{equation*}
	S' = -x^{-2},\qquad
	T' = -jx^{-j-1}e^x + x^{-j}e^x .
\end{equation*}
Now assume there is a ratio set $\ba$
such that $S' \fst^\ba T'$ for all $j \in \N$.  This would mean
\begin{equation*}
	\frac{x^{-2}}{x^{-j-1}e^x} = x^{j-1}e^{-x}
\end{equation*}
belongs to $\ba^+$ for all $j$, which is impossible since
$\ba^+$ is well-ordered for the reverse of $\fst$.
\end{ex}

\begin{pr}\label{genderiv}
Let $\bmu$ be a ratio set, and let $\ba$ be a derivative addendum
for $\bmu$  as
defined in Proposition~\ref{monoderivcompare}.
If $\bmu$ generates $T$ then $\ba$ generates $T'$.
If $\bmu$ generates and witnesses $T$ and $T \not\fe 1$, then
$\ba$ witnesses $T'$.
\end{pr}
\begin{proof}
Assume $\bmu$ generates $T$.
If $\m \in \supp T$, then $\m \in \GRID^\ebmu$, so
$\supp \m' \subseteq \GRID^\ba$ by Proposition~\ref{monoderivcompare}(b).
This holds for all $\m \in \supp T$, so
$\supp T' \subseteq \GRID^\ba$.  That is,
$\ba$ generates $T'$.

Now assume $\bmu$ generates and witnesses $T$ and $T \not\fe 1$.
Let $\g \in \supp(T')$.  Then
$\g \in \supp(\m')$ for some $\m \in \supp(T)$.
Now $\bmu$ witnesses $T$, so $\m \fsteq^\ebmu \mag(T)$.
Then by Proposition~\ref{monoderivcompare},
$\m' \fsteq^\ba (\mag T)' \sim \mag(T')$, so
$\m' \fsteq^\ba \mag(T')$ since $\ba$ witnesses $\m'$.
But $\g \in \supp(\m')$,
so $\g \fsteq^\ba \mag(T')$.
\end{proof}

\begin{ex}
The case $T \fe 1$ is not included in Proposition~\ref{genderiv}.
It is false:

Let $\bmu = \{x^{-1}, x^{-\sqrt{2}}\,\}$.  Then $\ba = \bmu$,
and
$$
	\bmu^* = \GRID^{\ebmu,\0} =
	\SET{x^{-j-k\sqrt{2}}}{j,k \in \N} .
$$
Let $T = 1 + x^{-1} + x^{-\sqrt{2}}$.  Then $\bmu$ witnesses
$T$, since $x^{-1}, x^{-\sqrt{2}} \in \bmu^*$.  So
$T' = -x^{-2}-\sqrt{2} x^{-1-\sqrt{2}} =
-x^{-2}(1+\sqrt{2}x^{1-\sqrt{2}}\,)$.
But $\bmu$ does not
witness $T'$ since $x^{1-\sqrt{2}} \not\in \bmu^*$.

Even more is true:  There is no ratio set $\ba$ such that
$\ba$ witnesses $T'$ for all $T$ witnessed by $\{x^{-1}, x^{-\sqrt{2}}\}$.
Indeed, $\{x^{-1}, x^{-\sqrt{2}}\}$ witnesses every transseries
$T = 1 + x^{-j} + x^{-k\sqrt{2}}$ with $j,k \in \N$, while
there exist pairs $(j,k) \in \N^2$ with $j-k\sqrt{2}$
negative but as close as we like to $0$.
\end{ex}

\begin{pr}\label{derivseriesgeom}
Let $\bmu$ be a ratio set, and let $\ba$ be a derivative addendum
for $\bmu$.  Assume series $\sum_{j=1}^\infty A_j$
is $\bmu$-geometrically convergent, $\bmu$ generates $A_1$,
and $A_1 \not\fe 1$.  Then
$\sum_{j=1}^\infty A_j'$ is $\ba$-geometrially convergent.
\end{pr}
\begin{proof}
Now $\bmu$ witnesses and generates all $A_j$, so $\ba$
witnesses $A_j'$.  If some $A_j \fe 1$, omit it.
Then $A_j' \fgt^\ba A_{j+1}'$, so $\sum A_j'$
is $\ba$-geometrically convergent.
\end{proof}

\begin{pr}\label{derivseriesgeommult}
Let $\bmu$ be a ratio set, and let $\ba$ be a derivative addendum
for $\bmu$.  Assume multiple series $\sum A_\bp$
is $\bmu$-geometrically convergent, $\bmu$ generates $A_\0$,
and $A_\0 \not\fe 1$.  Then
$\sum A_\bp'$ is $\ba$-geometrially convergent.
\end{pr}

\section{Composition}
Now we will consider a composition $T \circ S = T(S)$.  Here
$T,S \in \LP$ are large and positive.

Let $L$ be purely large (so that $\g = e^L$ is a monomial).  By \ref{D_exp},
a witness for $\g \circ S = e^{L\circ S}$ is
a witness for $\sm(L\circ S) \fst 1$.  A ratio set
for $e^{L\circ S}$ may be constructed as this witness together with
one more monomial $e^{\pm \la(L\circ S)}$.

\begin{de}\label{DCA1}
Let $\bmu = \{\mu_1,\cdots,\mu_n\}$ be a ratio set.
Write $\mu_i = e^{L_i}$, where $L_i$ is purely large
and negative.  For each $i$, let $\ba_i$ be a witness
for $\sm(L_i \circ S) \fst 1$.  Define $\ba = \bigcup_{i=1}^n \ba_i$.
(We use this definition only temporarily.)
\end{de}

\begin{de}\label{DCA}
Let $\bmu = \{\mu_1,\cdots,\mu_n\}$ be a ratio set.
Write $\mu_i = e^{L_i}$, where $L_i$ is purely large
and negative.  For each $i$, let $\ba_i$ be a witness
for $\sm(L_i \circ S) \fst 1$ and let
$\bb_i = \ba_i \cup \{e^{\la(L_i\circ S)}\}$.
Define $\bb = \bigcup_{i=1}^n \bb_i$.
The ratio set $\bb$ is called the
\Def{$S$-composition addendum} for $\bmu$.
\end{de}

Of course $\ba$ and $\bb$ depend on $\bmu$ and on $S$.  The
dependence on $S$ is not simply on a ratio set or
a witness for $S$, however.

\begin{re}
According to the construction given,
if $\bb$ is the $S$-composition addendum for $\bmu$,
then $\bb\circ\log := \SET{\fb\circ\log}{\fb \in \bb}$
is the $S\circ\log$-composition addendum for $\bmu$.
And $\bb\circ\exp$ is the $S\circ\exp$-composition addendum
for $\bmu$.  But in general it may not be true that
$\bb\circ U$ is the $S\circ U$-composition addendum
for $\bmu$.  The difference is that when
$L$ is purely large, $L\circ U$ need not be.
\end{re}

\begin{ex}
Suppose $\bmu \subset \G_0$.  Then $\mu_i = x^{b_i} = e^{b_i\log x}$.
Write $S = a e^A (1+U)$, with $a \in \R, a > 0, A \in \T, A>0$,
$A$ purely large, $U$ small.  Then
\begin{align*}
	&\log S = A + \log a + \sum_{j=1}^\infty \frac{(-1)^{j+1}}{j} U^j,
	\\
	&\la(L_i\circ S) = b_i A, \qquad
	e^{\la(L_i\circ S)} = e^{b_i A} = \mag(S)^{b_i},
	\\
	&\sm(L_i \circ S) = b_i \sum_{j=1}^\infty \frac{(-1)^{j+1}}{j} U^j .
\end{align*}
Now any witness for $S$ is a witness for $U \fst 1$, so a witness
for $\sm(L_i \circ S) \fst 1$.  So we may take $\ba$ any witness for $S$.
And $e^{\la(L_i\circ S)}=e^{b_iA}$ is a monomial.  So for $\bb$
add these $n$ monomials to $\ba$.
\end{ex}

\begin{ex}\label{G0_special}
\textit{A special case we need later.}  Not only $\bmu \subset \G_0$
but $S = x+B$ where $B \fst x$.  Then for $\ba$ we need a
witness for $S$, which is to say a witness for $B \fst x$.
And for $\bb$ we need to add $\mag(S)^{b_i} = x^{b_i} = \mu_i$.
So the $S$-composition addendum for $\bmu$ in this case is:
$\bmu$ itself together with a witness for $B \fst x$.
\end{ex}

\begin{pr}\label{compomono}
Let $\bmu$ be a ratio set, let $S \in \LP$,
let $\ba$ be as in Definition~\ref{DCA1},
and let $\bb$ be an $S$-composition addendum
as in Definition~\ref{DCA}.
Then {\rm(i)}~$\ba$ witnesses $\m(S)$ for all $\m \in \GRID^\ebmu$;
{\rm(ii)}~$\bb$ generates $\m(S)$ for all $\m \in \GRID^\ebmu$;
{\rm(iii)}~if $\m \in \G$ and $\m \fst^\ebmu 1$, then
$\m(S) \fst^\bb 1$;
{\rm(iv)}~if $\m,\n \in \G$ and $\m \fst^\ebmu \n$, then
$\m(S) \fst^\bb \n(S)$.
\end{pr}
\begin{proof}
For $1 \le i \le n$, we have $\mu_i = e^{L_i}$ and
$\mu_i\circ S = e^{L_i\circ S}$.  Then
$L_i \circ S = A+c+B$, where $A = \la(L_i\circ S)$
is purely large, $c$ is a constant, and $B = \sm(L_i \circ S)$
is small.  Of course $B \fst^\ba 1$ by the definition of $\ba$.
Then $\mu_i(S) = e^{A+c+B} = e^A e^c e^B$.  But
$e^A$ is a monomial, $e^c$ is a constant,
$$
	e^B = 1+\sum_{j=1}^\infty \frac{B^j}{j!} 
	\quad\text{and}\quad
	\sum_{j=1}^\infty \frac{B^j}{j!} \fst^\ba 1 .
$$
Therefore $\ba$ is a witness for $\mu_i(S)$.
By \ref{D_power} $\ba$ witnesses
$1/\mu_i(S)$.  By Proposition~\ref{sumproduct} $\ba$ witnesses
$\bmu^\bk(S)$ for all $\bk \in \Z^n$.  This proves (i).

Next note that $e^A \in \bb$ by the definition of $\bb$.
Therefore $\bb$ generates $\mu_i(S)$ for all $i$,
and $\bb$ generates $\bmu^\bk(S)$.  This proves (ii).
Also $e^A \fst^\bb 1$ by the definition of $\bb$,
so $\mu_i(S) \fst^\bb 1$.  By Proposition~\ref{sumproduct}
$\bmu^\bk(S) \fst^\bb 1$ for all $\bk > \0$.
This proves (iii).

Now assume $\m \fst^\ebmu \n$.  Then $\m/\n \fst^\bmu 1$.
By (iii), $(\m/\n)\circ S \fst^\bb 1$.
But $\bb$ witnesses $1$, so we may apply Proposition~\ref{productwitness}
to get $\big((\m/\n)\circ S\big)\cdot\big(\n\circ S\big) \fst^\bb \n\circ S$.
That is, $\m(S) \fst^\bb \n(S)$.  This proves (iv).
(Note: We did not assume $\m,\n \in \GRID^\ebmu$; we did not assume
that $\bb$ witnesses $\n\circ S$.)
\end{proof}

\begin{re}\label{composupp}
Consider a grid $\GRID^{\ebmu,\bm}$.  In the preceding
proposition, if $\n \in \GRID^{\ebmu,\bm}$,
then $\supp(\n\circ S) \subseteq \GRID^{\bb,\widetilde{\bm}}$,
where $\widetilde{\bm}$ is chosen so that
$\mag(\bmu^\bm \circ S) = \bb^{\widetilde{\bm}}$.
\end{re}

\begin{pr}\label{compconverge}
Let $\bmu$ be a ratio set, let $S \in \LP$, and let
$\bb$ be an $S$-composition addendum
as in Definition~\ref{DCA}.   Let $\sum_{i \in I} T_i$
be $\bmu$-convergent.  Then $\sum (T_i\circ S)$
is $\bb$-convergent.
\end{pr}
\begin{proof}
There is a grid $\GRID^{\ebmu,\bm}$ that supports all $T_i$, so
by Remark~\ref{composupp} there is a grid $\GRID^{\bb,\widetilde{\bm}}$
that supports all $T_i\circ S$.  So it remains to show that the series
$\sum (T_i\circ S)$ is point-finite.  Suppose, to the contrary, that there
is $\g$ such that $\AA = \SET{i \in I}{\g \in \supp(T_i\circ S)}$ is infinite.
For $i \in \AA$ there is $\n \in \supp(T_i)$ with
$\g \in \supp(\n\circ S)$.  Since $\sum T_i$ is point-finite,
there are infinitely many different $\n \in \bigcup\supp(T_i)$
with $\g \in \supp(\n\circ S)$.  This is contained in a grid
$\GRID^{\ebmu,\bm}$ so there is an infinite sequence
$\n_1 \fgt^\ebmu n_2 \fgt^\ebmu \cdots$ of such monomials.
But then by Proposition~\ref{compomono},
$\n_1 \circ S \fgt^\bb \n_2\circ S \fgt^\bb \cdots$.
So the sequence $\supp(\n_1\circ S),\supp(\n_2\circ S),\cdots$
is point-finite by \cite[\Emudominateprop]{edgar}.  So in fact
$\g$ cannot belong to all of them.  This contradiction
completes the proof.
\end{proof}

\begin{pr}\label{gencomp}
Let $\bmu$ be a ratio set, let $S \in \LP$
and let $\bb$ be as in Definition~\ref{DCA}.  Then
{\rm(i)}~If $\bmu$ generates $T$, then $\bb$ generates $T(S)$.
{\rm(ii)}~If $\bmu$ generates and witnesses $T$, then $\bb$ witnesses $T(S)$.
{\rm(iii)}~If $A \fst^\ebmu B$, $\bmu$ witnesses $B$,
and $\bmu$ generates $B$, then
$A(S) \fst^\bb \mag(B(S))$ so that $A(S) \fst^\bb B(S)$.
\end{pr}
\begin{proof}
(i) Let $\g \in \supp(T\circ S)$.  There is $\m \in \supp T$
with $\g \in \supp(\m \circ S)$.  Now $\m \in \GRID^\ebmu$,
so $\supp(\m\circ S) \subseteq \GRID^\bb$.

(ii) Write $T = a\g\cdot(1+U)$ be the
canonical multiplicative decomposition.  Then
$T(S) = a \g(S)\cdot(1+U(S))$.
Since $\bmu$ witnesses
$T$, we have $U \fst^\ebmu 1$.  So $U(S) \fst^\bb 1$ and
$\bb$ witnesses $1+U(S)$.  Since $\bmu$ generates $T$,
we have $\g \in \GRID^\ebmu$.  Therefore $\bb$
witnesses $\g(S)$.  So $\bb$ witnesses the product
$T(S) = a \g(S)\cdot(1+U(S))$.

(iii) Let $\g \in \supp  A(S)$.  There is $\m \in \supp(A)$
with $\g \in \supp \m(S)$.  Next, $A \fst^\bmu B$,
so there is $\n \in \supp(B)$ with $\m \fst^\ebmu \n$.
And $\ebmu$ witnesses $B$, so $\n \fsteq^\ebmu \mag(B)$.
Thus $\m \fst^\ebmu \mag(B)$.  Therefore
$\m(S) \fst^\bb \mag(B)(S)$ so there is
$\fb \in \supp(\mag(B)(S))$ with $\g \fst^\bb \fb$.  Now
$\bmu$ generates $B$, so $\mag(B) \in \GRID^\ebmu$,
so $\bb$ witnesses $\mag(B)(S)$.  So
$\fb \fsteq^\bb \mag(\mag(B)(S)) = \mag(B(S))$.
Thus $\g \fst^\bb \mag(B(S))$.  This shows
that $A(S) \fst^\bb B(S)$.
\end{proof}

\begin{pr}\label{compserieswitness}
Let $\bmu$ be a ratio set, let $S \in \LP$,
and let $\bb$ be an
$S$-composition addendum for $\bmu$ as
in Definition~\ref{DCA}.  Assume series $\sum_{j=1}^\infty A_j$
converges $\bmu$-geometrically and $\bmu$ generates $A_1$.
Then $\sum_{j=1}^\infty A_j(S)$ converges $\bb$-geometrically.
\end{pr}
\begin{proof}
Now $\bmu$ generates and witnesses all $A_j$, so $\bb$
generates and witnesses all $A_j(S)$.
And $A_j \fgt^\ebmu A_{j+1}$ so $A_j(S) \fgt^\bb A_{j+1}(S)$.
Therefore $\sum A_j(S)$ converges $\bb$-geometrically.
\end{proof}

\begin{pr}\label{compserieswitnessmult}
Let $\bmu$ be a ratio set, let $S \in \LP$,
and let $\bb$ be an
$S$-composition addendum for $\bmu$.  Assume
multiple series $\sum A_\bp$
converges $\bmu$-geo\-metric\-ally and $\bmu$ generates $A_\0$.
Then $\sum A_\bp(S)$ converges $\bb$-geometrically.
\end{pr}

\subsection*{Grid-Based Operator?}
Composition is not a ``grid-based operator'' of its right-hand
argument in the sense of \cite[p.~122]{hoeven}.

Consider
\begin{equation*}
	T = e^{-e^x},\qquad\qquad
	S = x + \sum_{j=1}^\infty a_j x^{-j} .
\end{equation*}
In fact, for our argument we will use only $a_j \in \{0,1\}$.

First let us compute $T \circ S$.  Writing $s = \sum_{j=1}^\infty a_j x^{-j}$,
we have
\begin{equation*}
	e^{S} = e^{x+s} = e^x\left(1+s+\frac{s^2}{2!}+\frac{s^3}{3!}+\cdots\right) ,
\end{equation*}
a transseries with support (contained in) $\SET{x^{-j}e^x}{j=0,1,\cdots}$.
So $e^S$ is purely large.  Next,
$T \circ S = e^{-e^S}$, which is a monomial.  For each subset
$E \subseteq \{1, 2, 3, \cdots \}$, if $S = x + \sum_{j \in E} x^{-j}$,
then we get a monomial $\m_E = T \circ S$.  Since logarithm exists for
transseries, the set $E$ can be recovered from $\m_E$, so there are
uncountably many monomials $\m_E$ of this kind.

Now what would it mean if $\Phi(Y) := T \circ (x+Y)$ were a grid-based
operator on $\R\lbb \MM \rbb$, where $\MM$ is a set of monomials
containing $x^{-j}, j \in \N$?  Say $\Phi = \sum_i \Phi_i$,
where $\Phi_i(Y) = \check{\Phi}_i(Y,Y,\cdots,Y)$ and
$\check{\Phi}_i$ is strongly $i$-linear.  So
\begin{align*}
	\Phi_i\left(\sum_{j \in E} x^{-j}\right) &=
	\sum_{j_1,\cdots,j_i \in E}
	\check{\Phi}_i\left(x^{-j_1},\cdots,x^{-j_i}\right),
	\\
	\Phi\left(\sum_{j \in E} x^{-j}\right) &=
	\sum_i \Phi_i\left(\sum_{j \in E} x^{-j}\right) ,
\end{align*}
and these are point-finite sums.  There are countably many terms
$\check{\Phi}_i\big(x^{-j_1},\cdots,x^{-j_i}\big)$, and each involves
only countably many monomials.  So since there are
uncountably many sets $E$, there are in fact
monomials $\m_E$ that are in none of these supports,
and thus is not in the support of any
$\Phi\big(\sum_{j \in E} x^{-j}\big)$.

\subsection*{Inverse}
Let $\bmu$ be a ratio set, let $S \in \LP$ and let
$T$ be inverse to $S$ so that $T \circ S = S \circ T = x$.
We would like ``composition addendum'' construction also
to be inverse.  It doesn't happen directly.  But
perhaps there is something almost as good.

\begin{qu}\label{invcompo}
Are there ratio sets $\ba, \bb$ so that
$\ba^* \supseteq \bmu$, $\bb$ is an $S$-composition addendum
for $\ba$ and $\ba$ is a $T$-composition addendum for $\bb$?
In particular: Using the construction of Definition~\ref{DCA},
let $\bb$ be composition addendum for $\bmu$, then
$\ba$ composition addendum for $\bb$.  Does it automatically
happen that $\bb$ is composition addendum for $\ba$?
If not two steps, does it stabilize in three?
\end{qu}

\section{Fixed Point}
The fixed point theorem in \cite[\Ecostinfixed]{edgar} (which
comes from Costin \cite{costintop} for example) uses
a ratio set $\bmu$ in an essential way.
And it was a main reason for the extent
of the use of ratio sets in that paper.  But here we will
discuss ``fixed point'' again.

Here is a ``geometric convergence'' version that is sometimes useful
but does not fit as a special case of \cite[\Ecostinfixed]{edgar}.

\begin{pr}\label{geometric_fixed}
Let $\SA \subseteq \T$,
let $\Phi \takes \SA \to \SA$ be a function,
and let $\ba$ be a ratio set.
Assume:
\begin{enumerate}
\item[{\rm(a)}]~if $\ba$ witnesses $S \in \SA$ then $\ba$ witnesses $\Phi(S)$;
\item[{\rm(b)}]~if $S,T \in \SA$ and $\ba$ witnesses $S-T$, then
$\ba$ witnesses $\Phi(S)-\Phi(T)$ and
$S-T \fgt^\ba \Phi(S)-\Phi(T)$;
\item[{\rm(c)}]~if $T_j \in \SA \;(j=1,2,\cdots)$
and $T_j$ converges $\ba$-geometrically
to $T$, then $T \in \SA$
\item[{\rm(d)}]~There exists $T_0 \in \SA$ such that
$\ba$ witnesses both $T_0$ and $\Phi(T_0)-T_0$.
\end{enumerate}
Then there is  $S \in \SA$ with $S = \Phi(S)$.
\end{pr}
\begin{proof}
First, choose $T_0 \in \SA$, using (d).
Then recursively define $T_{j+1} = \Phi(T_j)$ for $j \in \N$.
Now $\ba$ witnesses $T_0$ and $T_1-T_0$.
By (a), $\ba$ witnesses all $T_j$.
By (b), $\ba$ witnesses all $T_{j+1}-T_j$ and
$T_1-T_0 \fgt^\ba T_2-T_1 \fgt^\ba T_3-T_2 \fgt^\ba \cdots$.
So by Proposition~\ref{geometric} $T_j$
converges $\ba$-geometrically to some $S$.
So $S \in \SA$ and $\ba$ witnesses $S-T_j$ for all $j$.  Now
$(S-T_j)$ is point-finite, so by (c)
$(\Phi(S)-T_{j+1})$ is also point-finite, so $T_{j+1} \to \Phi(S)$.
Therefore $S = \Phi(S)$.
\end{proof}

The usual uniqueness proof does not work with these hypotheses.

\section{Witnessed Taylor's Theorem}
A simple version of Taylor's Theorem will approximate
$T(S+U)$ by $T(S)+T'(S)\cdot U$ when $U$ is small enough.
Under the right conditions, we should have
$T(S+U) -T(S) \sim T'(S)\cdot U$, see Theorem~\ref{taylor1}.
Here we want to consider a witnessed version of this.

Below we consider a condition
$\m(S)\cdot U \fst 1$ for all $\m \in \AA$, where $\AA$ is
a subgrid.  This may be written as
$(\AA\circ S)\cdot U \fst 1$.  Since a subgrid $\AA$ has a
maximum element $\m = \max \AA$, we can write
$(\AA \circ S)\cdot U \fst 1$ if and only if $(\m\circ S)\cdot U \fst 1$.
But the version with a witness will be
of the form $(\AA\circ S)\cdot U \fst^\bnu 1$, which is not equivalent
to $(\m\circ S)\cdot U \fst^\bnu 1$ unless $\bnu$ witnesses
$\AA \circ S$.

\subsection*{tsupp}
\begin{de}\label{de:tsupp}
We associate to each ratio set $\bmu$ a subgrid $\tsupp \bmu$.
[I was using $\lsupp \bmu$ for this at first, but it seems that
is not quite right.  I write here something that works in
the proofs, but perhaps it is sometimes larger than
really needed.]  This is defined recursively:

\noindent
(i) For non-monomials:
If $T \in \T$, then define $\tsupp T = \bigcup_{\g \in \supp T} \tsupp \g$,
and verify that it is a subgrid.

\noindent
(ii) For $b \in \R$, $b \ne 0$, define $\tsupp x^b = \{x^{-1}\}$;
$\tsupp 1 = \emptyset$.

\noindent
(iii) For $b \in \R$, $L \in \T_\bullet$ purely large, define
$\tsupp(x^be^L) = \supp(L') \cup \tsupp(L) \cup \{x^{-1}\}$.

\noindent
(iv) If $\tsupp$ has been defined on $\G_{\bullet,M}$, then define
it on $\G_{\bullet,M+1}$ by: $\tsupp(\g\circ \log)
= ((\tsupp \g)\circ \log)\cdot x^{-1} \cup \{x^{-1}\}$.

\noindent
(v) Sets: If $\AA \subseteq \T$, write
$\tsupp \AA = \bigcup_{\g \in \AA} \tsupp \g$.
\end{de}

\begin{ex} Compute:
$\tsupp(x^b) = \{x^{-1}\}$;
$\tsupp(e^{bx}) = \{1,x^{-1}\}$;
$\tsupp((\log x)^b) = \{(x \log x)^{-1},x^{-1}\}$.
\end{ex}

\begin{re}
Note that $x^{-1} \in \tsupp \bmu$ in every nontrivial case.
\end{re}

\begin{re}\label{tsuppprod}
If $\m,\n \in \G$, then $\tsupp(\m\n) \subseteq \tsupp \m \cup \tsupp \n$.
Also $\tsupp(1/\m) = \tsupp \m$.
\end{re}

\begin{re}
If $\AA$ is a subgrid, then there is a (finite!) ratio set
$\ba$ such that $\tsupp \AA = \tsupp \ba$.  Simply
choose $\ba$ so that
$\ba \subseteq \AA \subseteq \GRID^\ba$ and apply the following.
\end{re}

\begin{pr}
Let $\AA$ be a subgrid.  Then $\bigcup_{\g \in \AA} \tsupp \g$
is a subgrid.  If $\bmu$ is a ratio set, then
$\tsupp \GRID^\ebmu = \tsupp \bmu$.
\end{pr}
\begin{proof}
Since $\bmu \subseteq \GRID^\ebmu$ we have
$\tsupp \GRID^\ebmu \supseteq \tsupp \bmu$.
Write $\bmu = \{\mu_1,\cdots,\mu_n\}$.
If $\g \in \GRID^\ebmu$, then $\g = \bmu^\bk$ for
some $\bk$, so by Remark~\ref{tsuppprod} we
have $\tsupp \g \subseteq \bigcup_{i=1}^n \tsupp \mu_i
= \tsupp\bmu$.  So $\tsupp \GRID^\ebmu \subseteq \tsupp \bmu$.
\end{proof}

\begin{re}
If $\g \in \G_0$, then $\tsupp \g \subset \G_0$.
For $N \in \N$, $N \ge 1$: if $\g \in \G_N$,
then $\tsupp \g \subset \G_{N-1}$.  For $N,M \in \N$,
$N \ge 1, M \ge 1$: if $\g \in \G_{N,M}$, then
$\tsupp \g \subset \G_{\max(N-1,M),M}$.
If $\g$ is log-free, then $\tsupp \g$ is log-free.
If $\g$ has depth $M$, then $\tsupp \g$ has depth $M$.
\end{re}

\subsection*{Taylor Order 1}
Taylor's Theorem of order $1$ is the following:

\textit{Let $T, U_1, U_2 \in \T, S \in \LP$.
Assume $T \not\fe 1$,
$((\tsupp T) \circ S)\cdot U_1 \fst 1$, and
$((\tsupp T) \circ S)\cdot U_2 \fst 1$.
Then $S+U_1, S+U_2 \in \LP$ and
$T(S+U_1)-T(S+U_2) \sim T'(S) \cdot(U_1-U_2)$.}

This is proved below (Theorem~\ref{taylor1}).

\begin{ex}
Not valid with $\lsupp$ in place of $\tsupp$.
Let $T = \log x$, $S = x$, $U = x$.  So
$\lsupp T = \{T'/T\} = \{1/(x\log x)\}$.
And $(\lsupp T)\cdot U \fst 1$.  So
$$
	T(x+U)-T(x) = \log(2x)-\log(x) = \log 2,
	\qquad
	T'(x) U = \frac{x}{x} = 1,
$$
but $\log 2 \not\sim 1$.

Here $\tsupp T = \{1/(x\log x), 1/x\}$ so we would require
$U \fst x$.
\end{ex}

\begin{re}
Below note:  If $\AA \cdot U_1 \fst^\bb 1$, and
$\AA \cdot U_2 \fst^\bb 1$, then $\AA\cdot(U_1-U_2) \fst^\bb 1$.
Also note that we have not required that $U_1, U_2$ are witnessed
by $\bb$, only that they are generated by it, and their difference
is witnessed by it.
\end{re}

\subsection*{Special Case}
We will consider first the special case $S=x$ of
Taylor's Theorem of order $1$.
The special case is enough for the proof for the existence of
compositional inverses in Theorem~\ref{inv},
which is used in turn for a general case
of Taylor's Theorem.

\textit{Let $T, U_1, U_2 \in \T$.
Assume $T \not\fe 1$,
$(\tsupp T) \cdot U_1 \fst 1$, and
$(\tsupp T) \cdot U_2 \fst 1$.
Then $x+U_1, x+U_2 \in \LP$ and
$T(x+U_1)-T(x+U_2) \sim T'(x) \cdot(U_1-U_2)$.}

This is proved below (Theorem~\ref{sptaylor1}).
Here is the witnessed version of it.

\begin{thm}[Special Witnessed Taylor Order 1]\label{spwittaylor1}
Let $\bmu \subset \G$ be a ratio set.
Then there is a ratio set $\ba$ such that for all
ratio sets $\bb$ with $\bb^* \supseteq \ba$,
for all $T \in \TWG\ebmu\ebmu$ with $T \not\fe 1$,
and for all $U_1,U_2 \in \T^\bb$ with
$U_1-U_2 \in \TW\bb$ and
$$
	(\tsupp \bmu)\cdot U_1  \fst^\bb 1,\qquad
	(\tsupp \bmu)\cdot U_2  \fst^\bb 1 \;{\rm :}
$$
{\rm(a)}~$T(x+U_1)-T(x+U_2) \sim T'(x)\cdot(U_1-U_2)$.

\noindent{\rm(b)}~$\bb$ witnesses $T(x+U_1)-T(x+U_2)$.

\noindent{\rm(c)} $\bb$ generates $T(x+U_1)-T(x+U_2)$.

\noindent{\rm(d)}~If also $T \fst^\bmu x$ and $U_1 \ne U_2$, then
$$
	\frac{T(x+U_1)-T(x+U_2)}{U_1-U_2} \fst^\bb 1 .
$$
\end{thm}

This will be proved in several stages.

\begin{pr}\label{part_b}
In Theorem~\ref{spwittaylor1}, if $\bb$ satisfies
{\rm(a)} and {\rm(b)} and $\bb$ is a derivative addendum for $\bmu$, then
$\bb$ also satisies {\rm(c)} and {\rm(d)}.
\end{pr}
\begin{proof}
(c) From Proposition~\ref{genderiv}, since $\bmu$ generates $T$
we have $\bb$ generates $T'$.  Also $\bb$ generates $U_1$ and $U_2$,
so it generates $U_1-U_2$ and $T'(x)\cdot(U_1-U_2)$.
Therefore $\bb$ generates $T(x+U_1)-T(x+U_2)$.

(d) Assume $T \fst^\bmu x$.  Then $T' \fst^\bb 1$.
Since $\bb$ witnesses $U_1-U_2$, by Proposition~\ref{productwitness}
we get $T'(x)\cdot(U_1-U_2) \fst^\bb U_1-U_2$.
Since $\bb$ witnesses both $(T(x+U_1)-T(x+U_2))$ and $U_1-U_2$,
we conclude
$\bb$ witnesses $(T(x+U_1)-T(x+U_2))/(U_1-U_2)$.
Apply Proposition~\ref{wit_EL} to conclude
$(T(x+U_1)-T(x+U_2))/(U_1-U_2) \fst^\bb 1$.
\end{proof}

Write $\B[\AA,\bb,T]$ to mean:
For all $U_1,U_2 \in \T^\bb$, if
$U_1-U_2 \in \TW\bb$, 
$\AA  \cdot U_1 \fst^\bb 1$, and $\AA  \cdot U_2 \fst^\bb 1$, then
$\bb$ witnesses $T(x+U_1)-T(x+U_2)$ and
$T(x+U_1)-T(x+U_2) \sim T'(x)\cdot(U_1-U_2)$.

Write $\A[\bmu,\ba]$ to mean:
For all $\bb$ with $\bb^* \supseteq \ba$ and
for all $T \in \TWG\ebmu\ebmu$ with $T \not\fe 1$,
we have $\B[\tsupp \bmu,\bb,T]$.

So Theorem~\ref{spwittaylor1} says: for all $\bmu$ there
exists $\ba$ such that $\A[\bmu,\ba]$.

\begin{de}
Let $\bmu, \ba \subset \G_\bullet$ be a log-free ratio sets.
We say (recursively)
that $\ba$ is a \Def{Taylor addendum} for $\bmu$ if:

(a)~$\ba$ is a derivative addendum for $\bmu$;

(b)~for all $x^be^L \in \GRID^\bmu$ with $b\ne 0, L \ne 0$, we have
$x^{-1} \fst^\ba L'$;

(c)~$\ba$ is a Taylor addendum for $\tbmu$,
where $\tbmu$ is an exponent generator for $\bmu$.

\noindent
Begin the recursion by saying $\emptyset$ is a Taylor
addendum for $\emptyset$.
\end{de}

\begin{re}
If (c) holds for one exponent generator, then it also
holds for any other exponent generator, since they
generate the same subgroup of $\GRID^{\tebmu}$.
\end{re}

\begin{de}\label{de:tayloraddendum}
Let $\ba \subseteq \G_{\bullet,M}$ be a ratio set of logarithmic
depth $M$.  Then
$\widetilde{\bmu} = \bmu\circ\exp_M := \SET{\g\circ\exp_M}{\g \in \bmu}$
is a log-free ratio set.  We say that $\ba$
is a \Def{Taylor addendum} for $\bmu$ iff
$\ba\circ\exp_M$ is a Taylor addendum for $\widetilde{\bmu}$.
\end{de}

We will show: If $\ba$ is a Taylor addendum for $\bmu$,
then $\A[\bmu,\ba]$.

\begin{lem}\label{taylorlemma}
Let $\bmu \subset \G_\bullet$.  Then there is a
Taylor addendum for $\bmu$.
\end{lem}
\begin{proof}
Let $\bmu \subset \G_N$.  The proof is by induction on $N$.
For $N=0$, let $\ba$ be a derivative addendum for $\bmu$;
then (b) and (c) hold vacuously.

Assume $N > 0$ and the result holds for $N-1$.
Let $\tbmu$ be an exponent generator for
$\bmu$.  By the induction hypothesis, there is a Taylor addendum
$\widetilde{\ba}$ for $\tbmu$.
Write $\bmu = \{\mu_1,\cdots,\mu_n\}$,
$\mu_i = x^{b_i}e^{L_i}$, and
$$
	\SW := \SET{\sum_{i=1}^n p_i L_i}{\bp \in \Z^n} .
$$
So for any $x^b e^L \in \GRID^\ebmu$, we have $L \in \SW$.
The log-free exponent subgrid for $\GRID^\ebmu$ is
$\AA = \bigcup_{i=1}^n \supp L_i$ and
$\AA \subset \GRID^\tebmu \subset \G_{N-1}$.
From Lemma~\ref{supportlemma} there are only finitely many different
magnitudes in $\SW$:
$$
	\SET{\mag L}{L \in \SW} = \{\g_1,\cdots,\g_m\} .
$$
Let $\ba$ be a ratio set such that
$\ba^* \supseteq \widetilde{\ba}$,
$\ba$ is a derivative addendum for $\bmu$,
and for $1 \le i \le n$: 

$\g_i \fgt^\ba 1$,

$\ba$ witnesses $\AA_i := \SET{\m \in \AA}{\m \fsteq \g_i}$,

$\ba$ witnesses $\g_i'$,

$x^{-1} \fst^\ba \g_i'$.

\noindent
Such a ratio set exists since there are only finitely many requirements.
If $x^b e^L$ is any element of $\GRID^\ebmu$ with $L \ne 0$, then
$\mag L = \g_i$ for some $i$, so $L \subseteq \AA_i$ and
$L \fsteq^\ba \g_i$ so $\ba$ generates and witnesses $L$.
If $x^b e^L \in \GRID^\ebmu$, $b\ne 0$, $L\ne 0$,
then $x^{-1} \fst^\ba \g_i' \sim \mag(L')$,
so $x^{-1} \fst^\ba L'$.
\end{proof}

\begin{re}
Follow the construction to see: if $\bmu \subset \G_N$, then
the Taylor addendum $\ba$ may be chosen so that
$\ba \subset \G_N$.
\end{re}

\begin{pr}\label{WTsum}
Let $\bmu,\bb$ be ratio sets,
let $T \in \TWG\ebmu\ebmu$, $T \not\fe 1$, 
let $\AA \subset \G$, $x^{-1} \in \AA$.
Assume $\B[\AA,\bb,\g]$ for all $\g \in \supp T$.
Assume $\bb$ is a derivative addendum for $\bmu$.
Then $\B[\AA,\bb,T]$.
\end{pr}
\begin{proof}
In the proof of $\B[\AA,\bb,T]$,
if $T$ has a constant term it may be deleted, since that
changes neither the hypothesis nor the conclusion.
Let $U_1,U_2 \in \T^\bb$ with
$U_1-U_2 \in \TW\bb$, 
$\AA \cdot U_1 \fst^\bb 1$, and $\AA \cdot U_2 \fst^\bb 1$.
Then for any term $a\g$ of $T$:
$\bb$ witnesses $a\g(x+U_1)-a\g(x+U_2)$ and
\begin{equation*}
	a\g(x+U_1)-a\g(x+U_2) \sim a\g'\cdot(U_1-U_2) .
\tag{1}
\end{equation*}
Now the series $T = \sum a\g$ (considered as a multiple series
according to its grid, as in Remark~\ref{gridmultiseries})
converges $\bmu$-geometrically,
so $T' = \sum a\g'$ converges $\bb$-geometrically
by Proposition~\ref{derivseriesgeommult}.
So we may sum (1) using Lemma~\ref{multsummationlemma} to get:
$\bb$ witnesses $T(x+U_1)-T(x+U_2)$ and
$T(x+U_1)-T(x+U_2) \sim T'\cdot(U_1-U_2)$.
\end{proof}

\begin{pr}\label{WTG0}
Let $b \in \R$, $b \ne 0$, and let $\bb$ be a ratio set.
Then\hfill\break
$\B[\{x^{-1}\},\bb,x^b]$.
\end{pr}
\begin{proof}
Let $U_1,U_2 \in \T^\bb$.  Assume
$U_1\fst^\bb x$, $U_2 \fst^\bb x$, $U_1-U_2 \in \TW\bb$.  Then
\begin{align*}
	(x+U_1)^b-(x+U_2)^b &=
	x^b\sum_{j=1}^\infty \binom{b}{j}\left(
	\left(\frac{U_1}{x}\right)^j-\left(\frac{U_2}{x}\right)^j\right)
	\\ &=
	x^b\sum_{j=1}^\infty \binom{b}{j}\left(\frac{U_1-U_2}{x}\right)
	\sum_{k=0}^{j-1}\left(\frac{U_1}{x}\right)^{k}
	\left(\frac{U_2}{x}\right)^{j-1-k} .
\end{align*}
Now $\bb$ witnesses the fact that
each term ($j>1$) is $\fst$ the first term ($j=1$), and $\bb$ witnesses
that first term $(U_1-U_2)/x$.  So $\bb$ witnesses
the sum $(x+U_1)^b-(x+U_2)^b$ and
$(x+U_1)^b-(x+U_2)^b \sim x^{b-1}(U_1-U_2)$.
\end{proof}

\begin{co}\label{WTG0_A}
Let $\bmu \subset \Gsmall_0$ be a
ratio set.  Let $\ba$ be a ratio set such that
$\ba$ is a derivative addendum for $\bmu$.
Then $\A[\bmu,\ba]$.
\end{co}
\begin{proof}
Let $\bb$ be a ratio set with $\bb^* \supseteq \ba$.
Then $\bb$ is also a derivative addendum for $\bmu$.
Since $\tsupp \bmu = \{x^{-1}\}$,  for all $\g \in \GRID^\ebmu$ we have
$\B[\tsupp \bmu,\bb,\g]$.
So $\B[\tsupp \bmu,\bb,T]$ for all $T \in \TWG\ebmu\ebmu$
with $T \not\fe 1$ by Proposition~\ref{WTsum}.
This proves $\A[\bmu,\ba]$.
\end{proof}

\begin{pr}\label{WTlog}
Let $\bb$ be a ratio set.
Then
$\B[\{x^{-1}\},\bb,\log]$.
\end{pr}
\begin{proof}
Let $U_1,U_2 \in \T^\bb$.  Assume
$U_1\fst^\bb x$, $U_2 \fst^\bb x$, and $U_1-U_2 \in \TW\bb$.  Then
\begin{align*}
	\log(x+U_1)-\log(x+U_2) &=
	\log\left(1+\frac{U_1}{x}\right) - \log\left(1+\frac{U_2}{x}\right)
	\\&=
	\sum_{j=1}^\infty \frac{(-1)^{j+1}}{j}\left(
	\left(\frac{U_1}{x}\right)^j-\left(\frac{U_2}{x}\right)^j\right)
	\\ &=
	\sum_{j=1}^\infty \frac{(-1)^{j+1}}{j}\left(\frac{U_1-U_2}{x}\right)
	\sum_{k=0}^{j-1}\left(\frac{U_1}{x}\right)^{k}
	\left(\frac{U_2}{x}\right)^{j-1-k} .
\end{align*}
Now $\bb$ witnesses the fact that
each term ($j>1$) is $\fst$ the first term ($j=1$), and $\bb$ witnesses
that first term $(U_1-U_2)/x$.  So $\bb$ witnesses
the sum $\log(x+U_1)-\log(x+U_2)$ and
$\log(x+U_1)-\log(x+U_2) \sim (U_1-U_2)/x$.
\end{proof}

\begin{co}\label{WTlogcor}
In the preceding proof, if $\bb$ also generates $x$, then
$\bb$ generates $\log(x+U_1)-\log(x+U_2)$.
\end{co}
\begin{proof} Now $\bb$ generates $U_1$ and $U_2$, so it generates $U_1-U_2$.
If $\bb$ also generates $x$, then it generates $1/x$ and
$(U_1-U_2)/x$, and therefore $\bb$ generates $\log(x+U_1)-\log(x+U_2)$.
\end{proof}

\begin{pr}\label{WTexp}
Let $\bmu,\bb$ be ratio sets.
Let $b \in \R$ and let $L \in \TWG\ebmu\ebmu$ be purely large, $L \ne 0$.
Assume $\B[\AA,\bb,L]$ and  {\rm(}if $b \ne 0${\rm)} assume
$x^{-1} \fst^\bb L'$.
Then $\B[\{x^{-1}\}\cup\AA\cup\supp L',\bb,x^be^L]$.
\end{pr}
\begin{proof}
We take the case $b\ne 0$.  The case $b=0$ is similar but easier.
Write $\g = x^be^L$, so that $\g = e^{b\log x + L}$.
Let $U_1,U_2 \in \T^\bb$ with
$U_1-U_2 \in \TW\bb$, 
$(\{x^{-1}\}\cup\AA\cup\supp L')\cdot U_1 \fst^\bb 1$, and
$(\{x^{-1}\}\cup\AA\cup\supp L')\cdot U_2 \fst^\bb 1$.
Then $\widetilde{\bb}$
witnesses $L(x+U_1)-L(x)$, $L(x+U_2)-L(x)$, and $L(x+U_1)-L(x+U_2)$;
also $L(x+U_1)-L(x) \sim L'\cdot U_1$,
$L(x+U_2)-L(x) \sim L'\cdot U_2$,
and $L(x+U_1)-L(x+U_2) \sim L'\cdot(U_1-U_2)$.
By Proposition~\ref{WTlog}, $\bb$
witnesses $b\log(x+U_1)-b\log(x)$,
$b\log(x+U_2)-b\log(x)$, and $b\log(x+U_1)-b\log(x+U_2)$;
also $b\log(x+U_1)-b\log(x) \sim b U_1/x$,
$b\log(x+U_2)-b\log(x) \sim b U_2/x$, and
$b\log(x+U_1)-b\log(x+U_2) \sim b (U_1-U_2)/x$.
Let
\begin{align*}
	Q_1 &= b\log(x+U_1)+L(x+U_1)-b\log(x)-L(x) ,
	\\
	Q_2 &= b\log(x+U_2)+L(x+U_2)-b\log(x)-L(x) ,
	\\
	Q_1-Q_2 &= b\log(x+U_1)+L(x+U_1)-b\log(x+U_2)+L(x+U_2) .
\end{align*}
Since $x^{-1} \fst L'$, we have $Q_1 \sim L'\cdot U_1
\sim (bx^{-1}+L')\cdot U_1 \fst^\bb 1$.
Similarly $Q_2 \sim (bx^{-1}+L')\cdot U_2 \fst^\bb 1$
and $Q_1-Q_2 \sim (bx^{-1}+L')\cdot(U_1-U_2) \fst^\bb 1$.
Since $x^{-1} \fst^\bb L'$, we have
$\bb$ witnesses $Q_1$ so
$Q_1 \fst^\bb 1$.  Similarly $\bb$ witnesses $Q_2$, $Q_2 \fst^\bb 1$,
$\bb$ witnesses $Q_1-Q_2$, and $Q_1-Q_2 \fst^\bb 1$.  Now
$$
	e^{Q_1} - e^{Q_2} = \sum_{j=1}^\infty \frac{Q_1^j-Q_2^j}{j!}
	= (Q_1-Q_2) \sum_{j=1}^\infty \frac{1}{j!} 
	 \sum_{k=0}^{j-1} Q_1^k Q_2^{j-1-k},
$$
so $\bb$ witnesses $e^{Q_1} - e^{Q_2}$ and
$e^{Q_1} - e^{Q_2} \sim Q_1-Q_2$.
Then
\begin{align*}
	\g(x+U_1)-\g(x+U_2) &=
	e^{b\log(x+U_1)+L(x+U_1)} - e^{b\log(x+U_2)+L(x+U_2)}
	\\ &=
	e^{b\log x + L}\left(e^{Q_1}-e^{Q_2}\right) =
	x^b e^L \left(e^{Q_1}-e^{Q_2}\right) .
\end{align*}
Now $\bb$ witnesses $e^{Q_1}-e^{Q_2}$ and $x^b e^L$ is a monomial,
so $\bb$ witnesses $\g(x+U_1)-\g(x+U_2)$.  Continuing:
\begin{align*}
	\g(x+U_1)-\g(x+U_2) &= x^b e^L \left(e^{Q_1}-e^{Q_2}\right)
	\sim x^b e^L (Q_1-Q_2) 
	\\ &\sim x^be^L (bx^{-1}+L')\cdot(U_1-U_2)
	= \g'\cdot(U_1-U_2) .
\end{align*}
Therefore $\B[\{x^{-1}\}\cup\AA\cup\supp L',\bb,x^be^L]$.
\end{proof}

\begin{pr}\label{WTGN}
Let $\bmu \subset \Gsmall_\bullet$ be a log-free ratio set.
Let $\ba$ be a Taylor addendum for $\bmu$.
Then $\A[\bmu,\ba]$.
\end{pr}
\begin{proof}
Say $\bmu \subset \Gsmall_N$.  The proof is by induction on $N$.
The case $N=0$ is Corollary~\ref{WTG0_A}.  Now let $N>1$ and
assume the result holds for $N-1$.
Let $\tbmu \subset \Gsmall_{N-1}$ be an exponent generator for $\bmu$.
Then $\ba$ is a Taylor addendum for $\tbmu$.
So by the induction hypothesis, $\A[\tbmu,\ba]$.

Let $\bb$ be a ratio set with $\bb^* \supseteq \ba$.
Note that for all $x^be^L \in \GRID^\ebmu$,
we have $\tsupp \bmu \supseteq \{x^{-1}\} \cup \tsupp \tbmu \cup \supp L'$.
We have $\B[\tsupp \bmu,\bb,x^be^L]$ for all $x^be^L \in \GRID^\ebmu$
by Proposition~\ref{WTexp}.
Thus by Proposition~\ref{WTsum} we have
$B[\tsupp \bmu,\bb,T]$ for all $T \in \TWG\ebmu\ebmu$
with $T \not\fe 1$.  Therefore $\A[\bmu,\ba]$.
\end{proof}

\begin{pr}\label{WTlogind}
Let $T \in \T$, let $\ba$ be a ratio set, and
let $\AA\subset \G$.
Define
$$
	\BB = \frac{\AA\circ\log}{x} \cup \{x^{-1}\} ,\quad
	\bb = \ba\circ\log := \SET{\fa\circ\log}{\fa \in \ba}.
$$
Assume $\B[\AA,\ba,T]$.
Then $\B[\BB,\bb,T\circ\log]$.
\end{pr}
\begin{proof}
Let $U_1, U_2 \in \T^\bb$.  Assume
$U_1-U_2 \in \TW\bb$, $\BB\cdot U_1 \fst^\bb 1$, and
$\BB\cdot U_2 \fst^\bb 1$.  Now $x^{-1}$ in $\BB$,
so by Proposition~\ref{WTlog}, we conclude
$\bb$ witnesses $\log(x+U_1) - \log(x)$, $\log(x+U_2) - \log(x)$, and
$\log(x+U_1) - \log(x+U_2)$; and
$\log(x+U_1) - \log(x) \sim U_1/x$,
$\log(x+U_2) - \log(x) \sim U_2/x$, and
$\log(x+U_1) - \log(x+U_2) \sim (U_1-U_2)/x$.
Since $x^{-1} \in \bb$, by Corollary~\ref{WTlogcor}
we conclude that $\bb$ generates
$\log(x+U_1) - \log(x)$ and $\log(x+U_2) - \log(x)$.
Now define $V_1 := (\log(x+U_1) - \log(x))\circ\exp$ and
$V_2 := (\log(x+U_2) - \log(x))\circ\exp$, so that
$V_1,V_2 \in \TWG\ba\ba$, $V_1-V_2 \in \TW\ba$,
$V_1 \sim \big(U_1/x\big)\circ\exp$,
$V_2 \sim \big(U_2/x\big)\circ\exp$, and
and $V_1-V_2 \sim \big((U_1-U_2)/x\big)\circ\exp$.
By the definition of $\BB$ in terms of $\AA$, it follows
that $\AA\cdot V_1 \fst^\ba 1$ and $\AA\cdot V_2 \fst^\ba 1$.
We may apply $\B[\AA,\ba,T]$ to conclude
$\ba$ witnesses $T(x+V_1)-T(x+V_2)$ and
$T(x+V_1)-T(x+V_2) \sim T'\cdot(V_1-V_2)$.
Now
\begin{equation*}
	T\big(\log(x+U_1)\big) - T\big(\log(x+U_2)\big)
	=
	\big(T(x+V_1) - T(x+V_2)\big)\circ\log ,
\end{equation*}
so $\bb$ witnesses $T\big(\log(x+U_1)\big) - T\big(\log(x+U_2)\big)$.
Continuing,
\begin{align*}
	&T\big(\log(x+U_1)\big) - T\big(\log(x+U_2)\big)
	\sim
	\big(T'\cdot (V_1-V_2)\big)\circ \log
	\\ &\qquad\qquad\sim
	\frac{T'(\log x)\cdot (U_1-U_2)}{x}
	=
	(T\circ\log)'\cdot(U_1-U_2) .
\end{align*}
\end{proof}

\begin{co}\label{WTG}
Let $\bmu \subset \Gsmall$ be a ratio set.
Let $\ba$ be a Taylor addendum for $\bmu$.
Then $\A[\bmu,\ba]$.
\end{co}
\begin{proof}
By induction on $M$, where $\bmu \subset \G_{\bullet,M}$.
Apply Definitions \ref{de:tsupp} and~\ref{de:tayloraddendum}
using Propositions \ref{WTGN} and~\ref{WTlogind}.
\end{proof}

Together with Proposition~\ref{part_b},
this completes the proof of Theorem~\ref{spwittaylor1}.

Is the addendum $\bb$ constructed above much larger than necessary?

\begin{co}\label{GN_compo}
Let $\bmu,\bb \subset \Gsmall$ be ratio sets.
Let $B \in \T^\bb$.
Assume $\bb$ is a Taylor addendum for $\bmu$,
$\bb^* \supseteq \bmu$, and $(\tsupp \bmu)\cdot B \fst^\bb 1$.
Then $\bb$ is an $(x+B)$-composition addendum for $\bmu$.
\end{co}
\begin{proof}
Write $\bmu = \{\mu_1,\cdots,\mu_n\}$ and $\mu_i = e^{L_i}$.
Then $\supp L_i \subseteq \tsupp \bmu$, so $L_i' \cdot B \fst^\bb 1$.
Now $\tsupp L_i \subseteq \tsupp \bmu$, so we have
$\bb$ generates $L_i(x+B) - L_i$ and $L_i(x+B) - L_i
\sim L_i'\cdot B \fst^\bb 1$.
So $\sm(L_i\circ(x+B)) = L_i(x+B) - L_i$ and $\bb$ witnesses the fact
that this is ${}\fst 1$.
And $e^{\la(L_i\circ(x+B))} = e^{L_i} = \mu_i$
is witnessed by $\bb$.
\end{proof}

The non-witnessed version is a consequence.

\begin{thm}[Special Taylor Order 1]\label{sptaylor1}
Let $T, U_1, U_2 \in \T$.
Assume $(\tsupp T) \cdot U_1 \fst 1$, and
$(\tsupp T) \cdot U_2 \fst 1$.
Then $x+U_1, x+U_2 \in \LP$ and
$T(x+U_1)-T(x+U_2) \sim T'(x) \cdot(U_1-U_2)$.
\end{thm}
\begin{proof}
We may assume without loss of generality that $T\not\fe 1$,
since subtracting a constant from $T$ does not change
the conclusion.
Since $x^{-1} \in \tsupp T$, from $(\tsupp T)\cdot U_1 \fst 1$
we conclude $U_1 \fst x$.  Similarly $U_2 \fst x$.  So
$x+U_1, x+U_2 \in \LP$.  Let $\bmu$ be a ratio set with
$T \in \TWG\ebmu\ebmu$, and let $\ba$ be the
Taylor addendum for $\bmu$.  Choose $\bb \supseteq \ba$
such that $U_1, U_2 \in \T^\bb$, $U_1 - U_2 \in \TW\bb$,
and
$$
	(\tsupp \bmu)\cdot U_1  \fst^\bb 1,\qquad
	(\tsupp \bmu)\cdot U_2  \fst^\bb 1 .
$$
Then from Theorem~\ref{spwittaylor1}(a) we conclude
$T(x+U_1)-T(x+U_2) \sim T'(x) \cdot(U_1-U_2)$.
\end{proof}

\section{Compositional Inverse}
Notation: $\LP = \SET{S \in \T}{S \fgt 1, S > 0}$.
The set $\LP$ is a group inder the ``composition'' operation $\circ$.
We assume associativity is known.  The identity is $x \in \LP$.

\begin{thm}\label{inv}
Let $T \in \LP$.  Then there exists $S \in \LP$ with $T \circ S = x$.
\end{thm}

The proof proceeds in stages.  See \cite[\S~5.4.1]{hoeven},
\cite[Cor.~6.25]{DMM}.

\begin{pr}\label{inv_A0}
Let $A \in \T_0$, $A \fst x$.
Then there is $B \in \T_0$, $B \fst x$, so that
$(x+A)\circ(x+B) = x$.
\end{pr}
\begin{proof}
Write $\fa = \mag A$.  So $\fa \fst x$.
Let $\bmu \subseteq \G_0$ be a ratio set that generates $A$,
witnesses $A$, and witnesses $x+A$. In particular, $\fa \fst^\ebmu x$.
Now $\bb := \bmu \cup \{x^{-1}\} \subset \G_0$
is a Taylor addendum for $\bmu$ (Lemma~\ref{taylorlemma}
and Example~\ref{ex:derivativeaddendum}).
Let $\BB = \SET{\g \in \G_0}{\g \fsteq^\bb \fa}$ and
$\SD = \SET{B \in \TWG\bb\bb}{\supp B \subseteq \BB, B \sim -\fa}$.
Define $\Phi$ by
$$
	\Phi(B) := -A\circ(x+B) .
$$

I claim $\Phi$ maps $\SD$ into itself.
Indeed, let $B \in \SD$. Then by Example~\ref{G0_special}, $\bb$
is an $(x+B)$-composition addendum for $\bmu$.  But $\bmu$
generates and witnesses $A$, so $\bb$ generates and witnesses
$A\circ(x+B)$ by Proposition~\ref{gencomp}.
Note $\tsupp\bmu = \{x^{-1}\}$.
We have $B/x \fsteq^\bb \fa/x \fst^\bb 1$.
Then by Special Taylor~\ref{spwittaylor1} we have:
$\bb$ witnesses $A\circ(x+B)-A$ and $A\circ(x+B)-A \fst^\bb B \fsteq \fa$,
so $A\circ(x+B) \sim A \sim \fa$ and thus $\Phi(B) \in \SD$.
Therefore $\Phi$ maps $\SD$ into itself.

Note $T_0 := -A \in \SD$, $\Phi(T_0) \in \SD$,
$\bb$ witnesses $T_0$,
and---as just seen---$\bb$ witnesses $\Phi(T_0)-T_0$.

If $\bb$ witnesses $B \in \SD$, then $\bb$ witnesses $\Phi(B)$.

If $T_j \in \SD$ and $T_j$ converges geometrically to $T$,
then $T \in \SD$ by Proposition~\ref{geom_basic}.

Next let $B_1, B_2 \in \SD$ and assume $\bb$ witnesses $B_1-B_2$.
Then by Proposition~\ref{spwittaylor1} as above, we have:
$\bb$ witnesses $A\circ(x+B_1)-A\circ(x+B_2)$ and
$A\circ(x+B_1)-A\circ(x+B_2) \fst^\bb B_1-B_2$.  That is,
$\bb$ witnesses $\Phi(B_1)-\Phi(B_2)$ and
$\Phi(B_1)-\Phi(B_2) \fst^\bb B_1-B_2$.

We may now apply the fixed point theorem
Proposition~\ref{geometric_fixed} to conclude
there is $B \in \SD$ such that $B = \Phi(B)$.
That is:
$B = -A\circ(x+B)$ or
$x+B=x-A\circ(x+B)$ or
$x+B+A\circ(x+B) = x$ or
$(x+A)\circ(x+B) = x$.
\end{proof}

\begin{pr}\label{inv_AN}
Let $N \in \N$, $N \ge 1$.  Let $A \in \T$ with
$\supp A \subseteq \Gsmall_N \setminus \G_{N-1}$.
Then there is $B \in \T$ with
$\supp B \subseteq \Gsmall_N \setminus \G_{N-1}$
such that $(x+A)\circ(x+B) = x$.
\end{pr}
\begin{proof}
Write $\fa = \mag A \in \Gsmall_N \setminus \G_{N-1}$.
Let $\bmu \subset \Gsmall_N$ be a ratio set
that generates $A$, witnesses $A$, and witnesses $x+A$.
Now $\tsupp \bmu \subset \G_{N-1}$ so $(\tsupp \bmu)\cdot\fa \fst 1$.
Let $\bb$ be a Taylor addendum for $\bmu$ such that
$\bb^* \supseteq \bmu$ and
$(\tsupp \bmu)\cdot\fa \fst^\bb 1$.
Let $\BB = \SET{\g \in \G_N}{\g \fsteq^\bb \fa}$ and
$\SD = \SET{B \in \TWG\bb\bb}{\supp B \subseteq \BB, B \sim -\fa}$.
Define $\Phi$ by
$$
	\Phi(B) := -A\circ(x+B) .
$$

I claim $\Phi$ maps $\SD$ into itself.  Let $B \in \SD$.
Then by Corollary~\ref{GN_compo} $\bb$ is an
$(x+B)$-composition addendum for $\bmu$.
Since $\bmu$ generates and witnesses $A$, it follows
that $\bb$ generates and witnesses $A\circ(x+B)$.
By Special Taylor~\ref{spwittaylor1} we have
$\bb$ witnesses $A\circ(x+B) - A$ and
$A\circ(x+B) - A \fst^\bb B \sim -\fa$, so
$\bb$ witnesses $A\circ(x+B)$ and $A\circ(x+B) \sim A \sim \fa$.
Thus $\Phi(B) \sim -\fa$.  Also $\supp A\circ(x+B) \subseteq \Gsmall_N$
by \cite[\EcomponN]{edgar}, so $\Phi(B) \in \SD$.
Therefore $\Phi$ maps $\SD$ into itself.

Note $T_0 := -A \in \SD$, $\Phi(T_0) \in \SD$,
$\bb$ witnesses $T_0$,
and---as just seen---$\bb$ witnesses $\Phi(T_0)-T_0$.

If $\bb$ witnesses $B \in \SD$, then $\bb$ witnesses $\Phi(B)$.

If $T_j \in \SD$ and $T_j$ converges geometrically to $T$,
then $T \in \SD$ by Proposition~\ref{geom_basic}.

Next let $B_1, B_2 \in \SD$ and assume $\bb$ witnesses $B_1-B_2$.
Then by Proposition~\ref{spwittaylor1} as above, we have:
$\bb$ witnesses $A\circ(x+B_1)-A\circ(x+B_2)$ and
$A\circ(x+B_1)-A\circ(x+B_2) \fst^\bb B_1-B_2$.  That is,
$\bb$ witnesses $\Phi(B_1)-\Phi(B_2)$ and
$\Phi(B_1)-\Phi(B_2) \fst^\bb B_1-B_2$.

We may now apply the fixed point theorem
Proposition~\ref{geometric_fixed} to conclude
there is $B \in \SD$ such that $B = \Phi(B)$.
That is:
$B = -A\circ(x+B)$ or
$x+B=x-A\circ(x+B)$ or
$x+B+A\circ(x+B) = x$ or
$(x+A)\circ(x+B) = x$.
\end{proof}

\begin{pr}\label{inv_lfx}
Let $T \in \T_\bullet$.  Assume $T \sim x$.
Then there exists $S \in \T_\bullet$ with 
$S \sim x$ and $T\circ S = x$.
\end{pr}
\begin{proof}
Let $N \in \N$ be minimum so that $T \in \T_N$.  The
proof is by induction on $N$.  The case $N=0$ is
Proposition~\ref{inv_A0}.  Now assume $N \ge 1$
and the result is known for smaller values.

Now $T = x+A_0+A_1$, where $\supp A_0 \subset \G_{N-1}$,
$\supp A_1 \subset \Gsmall_N\setminus\G_{N-1}$, $A_0 \fst x$.
The induction hypothesis may be applied
to $x+A_0$, so there is $B_0$ with $\supp B_0 \subseteq \G_{N-1}$,
$B_0 \fst x$, and $(x+A_0)\circ(x+B_0)=x$.
Therefore $x+B_0+A_0\circ(x+B_0)=x$ so
$B_0+A_0\circ(x+B_0)=0$.

Write $C = A_1\circ(x+B_0)$ so that
$\supp C \subset \Gsmall_{N}\setminus\G_{N-1}$
by \cite[\EcomponN]{edgar}.
By Proposition~\ref{inv_AN} there is $D$ with
$\supp D \subset \Gsmall_{N}\setminus\G_{N-1}$
and $(x+C)\circ(x+D) = x$.
Let $E = D+B_0\circ(x+D)$
so that $\supp E \subset \G_N$ by \cite[\EcomponN]{edgar},
$E \fst x$,
and $x+E = (x+B_0)\circ(x+D)$.
Let $S = x + E$.
\begin{align*}
	x &= (x+C)\circ(x+D)=\big(x+0+A_1\circ(x+B_0)\big)\circ(x+D)
	\\ &=
	\big(x+B_0+A_0\circ(x+B_0)+A_1\circ(x+B_0)\big)\circ(x+D)
	\\ &=
	(x+A_0+A_1)\circ(x+B_0)\circ(x+D) = T \circ S .
\end{align*}
with $S = x+E \sim x$.
\end{proof}

\begin{pr}\label{inv_x}
Let $T \in \T$.  Assume $T \sim x$.
Then there exists $S \in \T$ with 
$S \sim x$ and $T\circ S = x$.
\end{pr}
\begin{proof}
If $T$ is log-free, this follows from Proposition~\ref{inv_lfx}.
If $T \in \T_{\bullet,M}$, then
$T_1 = \log_M \circ T \circ \exp_M \in \T_\bullet$ and still
$T_1 \sim x$.  So there is $S_1$ with $T_1 \circ S_1 = x$.
Then $S = \exp_M \circ S_1 \circ \log_M$
satisfies $T \circ S = x$.
\end{proof}

\begin{pr}\label{inv_leaf}
Let $M \in \Z$.  Let $T \in \T$ with $T \sim \l_M$.
Then there exists $S \in \LP$ with $T\circ S = x$.
\end{pr}
\begin{proof}
Let $T_1 = T \circ \exp_M$, so that $T_1 \sim x$.
Then there is $S_1$ with $T_1 \circ S_1 = x$.
So $S = \exp_M \circ S_1$ satisfies
$T \circ S = x$.
\end{proof}

\begin{proof}[Proof of Theorem~\ref{inv}]
Let $T \in \LP$.
There exist $m,p$ so that $\log_m \circ T \sim \l_p$
(\cite[\Cexponentiality]{edgarc}).
But there exists $S_1$ with
$(\log_m \circ T)\circ S_1 = x$.  Then
$S = S_1 \circ \exp_m$ satisfies
$T \circ S = x$.
\end{proof}

\begin{re}
As is well-known: if right inverses all exist, then they are
full inverses.  Review of the proof:
Suppose $T \circ S = x$ as found.  Start with $S$ and
get a right-inverse $T_1$ so $S \circ T_1 = x$.
Then
\begin{equation*}
	T = T \circ x
	= T \circ (S \circ T_1)
	= (T \circ S) \circ T_1
	= x \circ T_1 = T_1 .
\end{equation*}
\end{re}

\begin{no}
Write $T^{[-1]}$ for the compositional inverse of $T$.
\end{no}

\subsection*{Taylor's Theorem Again}
The general order one Taylor's Theorem is deduced from
the case $\sim x$ using a compositional inverse.

\begin{thm}[Taylor Order 1]\label{taylor1}
Let $T, U_1, U_2 \in \T, S \in \LP$.
Assume $((\tsupp T) \circ S)\cdot U_1 \fst 1$, and
$((\tsupp T) \circ S)\cdot U_2 \fst 1$.
Then $S+U_1, S+U_2 \in \LP$ and
$T(S+U_1)-T(S+U_2) \sim T'(S) \cdot(U_1-U_2)$.
\end{thm}
\begin{proof}
Because $S$ has an inverse, there exist
$\widetilde{U}_1,\widetilde{U}_2$ such that
$\widetilde{U}_1\circ S = U_1$ and $\widetilde{U}_2 \circ S = U_2$.
Then
\begin{align*}
	((\tsupp T) \circ S)\cdot U_1 \fst 1
	&\Longleftrightarrow
	((\tsupp T) \circ S)\cdot (\widetilde{U}_1\circ S) \fst 1
	\\ 	&\Longleftrightarrow
	(\tsupp T)\cdot \widetilde{U}_1 \fst 1 .
\end{align*}
Similarly $(\tsupp T)\cdot \widetilde{U}_2 \fst 1$.
Therefore by Theorem~\ref{sptaylor1},
$x+\widetilde{U}_1,x+\widetilde{U}_2 \in \LP$ and
$T(x+\widetilde{U}_1)-T(x+\widetilde{U}_2)
\sim T'(x)\cdot(\widetilde{U}_1-\widetilde{U}_2)$.
Compose on the right with $S$ to get
$S+U_1, S+U_2 \in \LP$ and 
$T(S+U_1)-T(S+U_2) \sim T(S)\cdot(U_1-U_2)$.
\end{proof}

\begin{qu}The witnessed version should be something like this:

\textit{Let $\bmu$ be a ratio set, and let $S \in \LP$.
Then there is a ratio set $\ba$ such that:
for all ratio sets $\bb$ with $\bb^* \supseteq \ba$,
for all $T \in \TWG\ebmu\ebmu$ with $T \not\fe 1$,
and for all $U_1, U_2 \in \T^\bb$ with
$U_1-U_2 \in \TW\bb$, and
$$
	((\tsupp \bmu) \circ S)\cdot U_1 \fst^\bb 1,\qquad
	((\tsupp \bmu) \circ S)\cdot U_2 \fst^\bb 1 :
$$
}

{\rm(a)}~$T(S+U_1)-T(S+U_2) \sim T'(S)\cdot(U_1-U_2)$.

{\rm(b)}~\textit{$\bb$ witnesses $T(S+U_1)-T(S+U_2)$.}

\textit{{\rm(c)}~$\bb$ generates $T(S+U_1)-T(S+U_2)$.}

\textit{{\rm(d)}~If also $T \fst^\bmu x$ and $U_1 \ne U_2$, then
$$
	\frac{T(S+U_1)-T(S+U_2)}{U_1-U_2} \fst^\bb 1 .
$$
}
\noindent
But deducing this from the special case
in Theorem~\ref{spwittaylor1} would require a positive
answer to Question~\ref{invcompo}.  If that doesn't
work out, then perhaps adapting the proof above
(\ref{part_b} through \ref{WTG}) would be required.
\end{qu}

\section{Mean Value Theorem}
Consider \cite[\Cderivcompare]{edgarc}:
\textit{Let $A, B \in \T$, $S_1,S_2 \in \P$, $A' \fst B'$,
$S_1 < S_2$.  Then}
\begin{equation*}
	A \circ S_2 - A \circ S_1 \fst 
	B \circ S_2 - B \circ S_1 .
\end{equation*}
Let us consider witnessed versions of it.

\subsection*{Fixed Upper Term}
\begin{pr}\label{fuppermono}
Let $\fb \in \G$, $\fb \ne 1$, $S_1, S_2 \in \P$, $S_1 < S_2$ be given.
Let $\bmu = \{\mu_1,\cdots,\mu_n\}$ be a ratio set.
Then there is a ratio set $\ba$ such that:
for every $\fa \in \G$, if $\bmu$ witnesses $\fa \fst \fb$,
then $\ba$ witnesses $\fa(S_2)-\fa(S_1) \fst \fb(S_2)-\fb(S_1)$.
\end{pr}
\begin{proof}
First, $\fb = e^B$ where $B$ is purely large and nonzero.  So
$B \fgt 1$.  Each $\mu_i \fst 1 \fst B$.  By \cite[\Cderivcompare]{edgarc},
for each $i$ we have
\begin{equation*}
	\mu_i(S_2) - \mu_i(S_1) \fst B(S_2) - B(S_1) .
\tag{1}
\end{equation*}
Next I claim
\begin{equation*}
	\fb(S_1)\big(\mu_i(S_2) - \mu_i(S_1)\big) \fst \fb(S_2)-\fb(S_1) .
\tag{2}
\end{equation*}
We take two cases.

\textit{Case 1.} $\fb(S_1) \fgt \fb(S_2)$.  Then
$\fb(S_2)-\fb(S_1) \sim \fb(S_1)$, $\mu_i(S_1) \fst 1$,
$\mu_i(S_2) \fst 1$, so we have
$$
	\fb(S_1)\big(\mu_i(S_2) - \mu_i(S_1)\big) \fst
	\fb(S_1) \sim \fb(S_2)-\fb(S_1) .
$$

\textit{Case 2.} $\fb(S_1) \fsteq \fb(S_2)$.  If
$B(S_2) > B(S_1)$ then (since $\exp$ is increasing)
$\fb(S_2) > \fb(S_1)$ and
\begin{align*}
	\fb(S_1)\big(B(S_2)-B(S_1)\big)
	&= \fb(S_1)\log\big(\fb(S_2)/\fb(S_1)\big)
	< \fb(S_1)\left(\frac{\fb(S_2)}{\fb(S_1)}-1\right)
	\\ &=\fb(S_2) - \fb(S_1) ,
\end{align*}
and both extremes are positive, so combining this with (1) we
get (2).
On the other hand, if $B(S_2) < B(S_1)$, then
\begin{align*}
	\fb(S_1)\big(B(S_1)-B(S_2)\big)
	&= \fb(S_1)\log\big(\fb(S_1)/\fb(S_2)\big)
	< \fb(S_1)\left(\frac{\fb(S_1)}{\fb(S_2)}-1\right)
	\\ & =\frac{\fb(S_1)}{\fb(S_2)}\big(\fb(S_1) - \fb(S_2)\big)
	\fsteq \fb(S_1) - \fb(S_2) ,
\end{align*}
and both extremes are positive, so combining this with (1) we
get (2).  This completes the proof of (2).

Now let the ratio set $\ba$ be such that:
for each $i$, $\ba$ witness $\mu_i(S_1) \fst 1$, 
$\mu_i(S_2) \fst 1$, and (2).  Such $\ba$ exists
because this is only a finite list of requirements.

Now let $\fa \in \G$ and let
$\bmu$ witness $\fa \fst \fb$.  We must show that
$\ba$ witnesses $\fa(S_2)-\fa(S_1) \fst \fb(S_2)-\fb(S_1)$.
Now $\fa = \fb\g_1\g_2\cdots\g_J$,
where $\g_j \in \bmu$ for all $j$ and $J \ge 1$.
Compute
{\allowdisplaybreaks
\begin{align*}
	\fa(S_2) - \fa(S_1)
	&= \fb(S_2)\prod_{j=1}^J \g_j(S_2) - \fb(S_1)\prod_{j=1}^J \g_j(S_2)
	\\ &= \big(\fb(S_2)-\fb(S_1)\big)\prod_1^J\g_j(S_2)
	\\ &\quad+
	\fb(S_1)\big(\g_1(S_2)-\g_1(S_1)\big)\prod_2^J\g_j(S_2)
	\\ &\quad+
	\fb(S_1)\g_1(S_1)\big(\g_2(S_2)-\g_2(S_1)\big)\prod_3^J\g_j(S_2)
	\\ &\quad+\dots
	\\ &\quad+
	\fb(S_1)\prod_1^{k-1}\g_j(S_1)\big(\g_k(S_2)-\g_k(S_1)\big)
	\prod_{k+1}^J\g_j(S_2)
	\\ &\quad+\dots
	\\ &\quad+
	\fb(S_1)\prod_1^{J-1}\g_j(S_1)\big(\g_J(S_2)-\g_J(S_1)\big) .
\end{align*}
}
Finally note that $\ba$ witnesses that each of these terms is
$\fst \fb(S_2)-\fb(S_1)$: Each term has one or more factors
$\g_j(S_1) \fst^\ba 1$ or $\g_j(S_2) \fst^\ba 1$,
and $\ba$ witnesses $1$, so we may apply
Proposition~\ref{prodwitnessmult} even if
$\ba$ does not witness $\fb(S_2)-\fb(S_1)$.
\end{proof}

\begin{co}
Let $B \in \T$, $B \not\fe 1$, $S_1, S_2 \in \P$, $S_1 < S_2$ be given.
Let $\bmu = \{\mu_1,\cdots,\mu_n\}$ be a ratio set.
Then there is a ratio set $\bnu$ such that:
for every $A \in \T$, if $\bmu$ witnesses both $B$ and $A \fst B$,
then $\bnu$ witnesses $A(S_2)-A(S_1) \fst B(S_2)-B(S_1)$.
\end{co}
\begin{proof}
Let $\fb = \mag B$, so $\fb \ne 1$.  Let $\ba$ be the ratio set
of Proposition~\ref{fuppermono}.  Let
$\bb$ witness $\fb(S_2)-\fb(S_1)$.  Let $\bnu = \ba \cup \bb$.

Let $A$ be such that $\bmu$
witnesses both $B$ and $A \fst B$.  Now if
$\g \in \supp(A(S_2)-A(S_1))$, then there is $\fa \in \supp A$
with $\g \in \supp(\fa(S_2)-\fa(S_1))$.  But then
there is $\fb_0 \in \supp(B)$ with
$\fa \fst^\ebmu \fb_0 \fsteq^\ebmu \fb$, so by
Proposition~\ref{fuppermono} there is
there is $\m \in \supp(\fb(S_2)-\fb(S_1))$ with
$\g \fst^\bnu \m$.  And
$\m \fsteq^\bnu \mag(\fb(S_2)-\fb(S_1)) = \mag(B(S_2)-B(S_1))$.
This shows that $A(S_2)-A(S_1) \fst^\bnu B(S_2)-B(S_1)$.
\end{proof}

\begin{re}\label{Wfuppermonoc}
The particular case $\fb = x$ appears in \cite[\Cmvti]{edgarc}.
The construction for $\bnu$ from $\bmu$ in that case:
Let $\bmu = \{\mu_1,\cdots,\mu_n\}$ and $S_1 < S_2$ be given.
For each $i$, let $\ba_i$ witness:
$$
	\mu_i(S_1) \fst 1, \qquad
	\mu_i(S_2) \fst 1, \qquad
	\mu_i(S_2) - \mu_i(S_1) \fst \log S_2-\log S_1 .
$$
Then $\bnu = \bigcup_{i=1}^n \ba_i$ satisfies:
if $A \in \T$ and $\bmu$ witnesses $A \fst x$, then $\bnu$
witnesses $A(S_2)-A(S_1) \fst S_2-S_1$.

Also, since $x$ is increasing, (2) suffices, so we
could replace
$$
	\mu_i(S_2) - \mu_i(S_1) \fst \log S_2-\log S_1
	\qquad\text{by}\qquad
	\mu_i(S_2)-\mu_i(S_1) \fst \frac{S_2}{S_1} - 1.
$$
\end{re}

\subsection*{General Upper Term}
\begin{thm}\label{MVT}
Let $\bmu \subset \Gsmall$ be a ratio set.  Let $S_1, S_2 \in \LP$
with $S_1 < S_2$.  Then there is a ratio set $\ba$ such that:

{\rm(a)}~If $\fa,\fb \in \GRID^\ebmu$, $\fa \fst^\ebmu \fb$,
and $\fb \ne 1$,
then $\fa(S_2)-\fa(S_1) \fst^\ba \fb(S_2)-\fb(S_1)$.

{\rm(b)}~If $\g \in \GRID^\ebmu$ and $\g \fst^\ebmu 1$,
then $\ba$ witnesses $\g(S_2)-\g(S_1)$.

{\rm(b$'$)}~If $\g \in \GRID^\ebmu$ and $\g \fgt^\ebmu 1$,
then $\ba$ witnesses $\g(S_2)-\g(S_1)$.

{\rm(c)}~If $B \in \TWG\ebmu\ebmu$, $\fb = \mag B$, and $\fb \ne 1$,
then $\fb(S_2)-\fb(S_1) \fe B(S_2)-B(S_1)$.

{\rm(d)}~If $B \in \TWG\ebmu\ebmu$ and $B \fst^\ebmu 1$,
then $\ba$ witnesses $B(S_2)-B(S_1)$.

{\rm(e)}~If $A, B \in \TWG\ebmu\ebmu$, $A \fst^\ebmu B \fst^\ebmu 1$, then
$A(S_2)-A(S_1) \fst^\ba B(S_2)-B(S_1)$.

{\rm(f)}~If $\sum A_j$ converges $\bmu$-geometrically and $A_1 \fst^\ebmu 1$,
then $\sum(A_j(S_2)-A_j(S_1))$ converges $\ba$-geometrically

{\rm(g)}~If the multiple series $\sum A_\bp$ converges $\bmu$-geometrically
and $A_\0 \fst^\ebmu 1$, then $\sum(A_\bp(S_2)-A_\bp(S_1))$
converges $\ba$-geometrically.

\end{thm}
\begin{proof}
Write $\bmu = \{\mu_1,\cdots,\mu_n\}$ with $\mu_i = e^{L_i}$,
and $L_i$ is purely large.
By the Support Lemma~\ref{supportlemma}, the set
$$
	\SW := \SET{\sum_{i=1}^n p_i (L_i(S_2)-L_i(S_1))}{\bp \in \Z^n}
	\setminus \{0\}
$$
has finitely many different magnitudes:
$\SET{\mag Q}{Q \in \SW} = \{\g_1,\cdots,\g_m\}$.
If $1 \le j \le m$, then $\g_j = \mag Q$ with
$Q = \sum_{i=1}^n p_i (L_i(S_2)-L_i(S_1))$.
So for $1 \le l \le n$, since
$\sum p_iL_i \fgt 1 \fgt \mu_l$ we have
\begin{equation*}
	\mu_l(S_2) - \mu_l(S_1) \fst \sum p_i (L_i(S_2)-L_i(S_1)) \sim \g_j
\end{equation*}
by \cite[\Cderivcompare]{edgarc}.

Let the ratio set $\ba$
be such that:

$\ba$ witnesses $\mu_i(S_1)$;

$\ba$ witnesses $\mu_i(S_2)$;

$\ba$ witnesses $\mu_i(S_1)-\mu_i(S_2)$;

$\mu_i(S_1) \fst^\ba 1$;

$\mu_i(S_2) \fst^\ba 1$;

$\mu_i(S_1) - \mu_i(S_2) \fst^\ba \g_j$ for all $i,j$;

if $1 \le i,k \le n$ and
$\mu_i(S_1)-\mu_i(S_2) \fgt \mu_k(S_1)-\mu_k(S_2)$,
then
$$
	\mu_i(S_1)-\mu_i(S_2) \fgt^\ba \mu_k(S_1)-\mu_k(S_2) .
$$

Now let $\bmu^\bp \in \GRID^\ebmu$ with $\bmu^\bp \ne 1$.
Then $\bmu^\bp(S_1) \ne \bmu^\bp(S_2)$ and
$\log\bmu^\bp(S_2) - \log\bmu^\bp(S_1) \in \SW$,
so $\mag(\log\bmu^\bp(S_2) - \log\bmu^\bp(S_1)) = \g_j$
for some $j$.  Then for $1 \le i \le n$ we have
$\mu_i(S_2) - \mu_i(S_1) \fst^\ba \g_j$, so
of course
\begin{equation*}
	\mu_i(S_2) - \mu_i(S_1) \fst^\ba
	\log\bmu^\bp(S_2) - \log\bmu^\bp(S_1) .
\tag{1}
\end{equation*}
Write $V = \bmu^\bp(S_2)/\bmu^\bp(S_1)$
and note $V>0, V \ne 1$.  Since $\ba$ witnesses
$\mu_i(S_1)$ and $\mu_i(S_2)$ for all $i$,
by Propositions \ref{witproduct} and ~\ref{D_power},
$\ba$ witnesses $V$.
Next I claim
\begin{equation*}
	\bmu^\bp(S_1)\cdot\big(\mu_i(S_2)-\mu_i(S_1)\big)
	\fst^\ba
	\bmu^\bp(S_2) - \bmu^\bp(S_1) ,
\tag{2}
\end{equation*}
or equivalently $\mu_i(S_2)-\mu_i(S_1) \fst^\ba V-1$.
We prove this in five cases.

\textit{Case 1.} $V \sim 1$.
Then since $\ba$ witnesses $V$, we have $V-1 \fst^\ba 1$, so
\begin{equation*}
	\log V = \log\big(1+(V-1)\big) =
	\sum_{j=1}^\infty \frac{(-1)^{j+1}}{j}\,(V-1)^j
	\fst^\ba V-1 .
\end{equation*}
So by (1) we have $\mu_i(S_2)-\mu_i(S_1) \fst^\ba \log V \fst^\ba V-1$,
and therefore $\mu_i(S_2)-\mu_i(S_1) \fst^\ba V-1$.

\textit{Case 2.} $V \sim c, c \in \R, c>0, c \ne 1$.
Then
\begin{equation*}
	\mu_i(S_2)-\mu_i(S_1) \fst^\ba 1 \fe c-1 \fsteq^\ba V-1,
\end{equation*}
so by Proposition~\ref{wit_LE} we have
$\mu_i(S_2)-\mu_i(S_1) \fst^\ba V-1$.

\textit{Case 3.} $V \fst 1$.
Then $\mu_i(S_2)-\mu_i(S_1) \fst^\ba (-1) \sim V-1$, so
by Proposition~\ref{wit_LE} we have
$\mu_i(S_2)-\mu_i(S_1) \fst^\ba V-1$.
(Note: We do not say $\ba$ witnesses $V-1$.)

\textit{Case 4.} $V \fgt 1, \const V = 0$.  Then
$1 \in \supp(V-1)$, so $\mu_i(S_2)-\mu_i(S_1) \fst^\ba 1 \fsteq^\ba V-1$.
Thus $\mu_i(S_2)-\mu_i(S_1) \fst^\ba V-1$.
(Again in this case: We do not say $\ba$ witnesses $V-1$:
See Remark~\ref{case4counter}.)

\textit{Case 5.} $V \fgt 1, \const V \ne 0$.  Since $\ba$
witnesses $V$, this means $1 \fst^\ba \mag V = \mag(V-1)$.
So $\mu_i(S_2)-\mu_i(S_1) \fst^\ba 1 \fst^\ba V-1$.
Thus $\mu_i(S_2)-\mu_i(S_1) \fst^\ba V-1$.

This completes the proof of (2).

(a) Now let $\fa, \fb \in \GRID^\ebmu$ with $\fa \fst^\ebmu \fb$
and $\fb \ne 1$. We must show that
$\fa(S_2)-\fa(S_1) \fst^\ba \fb(S_2)-\fb(S_1)$.
Now $\fa = \fb\g_1\g_2\cdots\g_J$,
where $\g_j \in \bmu$ for all $j$ and $J \ge 1$.
Compute
\begin{align*}
	\fa(S_2) - \fa(S_1)
	&= \fb(S_2)\prod_{j=1}^J \g_j(S_2) - \fb(S_1)\prod_{j=1}^J \g_j(S_2)
	\\ &= \big(\fb(S_2)-\fb(S_1)\big)\prod_1^J\g_j(S_2)
	\\ &\quad+
	\fb(S_1)\big(\g_1(S_2)-\g_1(S_1)\big)\prod_2^J\g_j(S_2)
	\\ &\quad+
	\fb(S_1)\g_1(S_1)\big(\g_2(S_2)-\g_2(S_1)\big)\prod_3^J\g_j(S_2)
	\\ &\quad+\dots
	\\ &\quad+
	\fb(S_1)\prod_1^{k-1}\g_j(S_1)\big(\g_k(S_2)-\g_k(S_1)\big)
	\prod_{k+1}^J\g_j(S_2)
	\\ &\quad+\dots
	\\ &\quad+
	\fb(S_1)\prod_1^{J-1}\g_j(S_1)\big(\g_J(S_2)-\g_J(S_1)\big) .
\end{align*}
Note that $\ba$ witnesses that each of these terms is
$\fst \fb(S_2)-\fb(S_1)$; for this apply (2) in all terms
except the first.  Each term has one or more factors
$\g_j(S_1) \fst^\ba 1$ or $\g_j(S_2) \fst^\ba 1$,
and $\ba$ witnesses $1$, so we may apply
Proposition~\ref{prodwitnessmult} even if
$\ba$ does not witness $\fb(S_2)-\fb(S_1)$.
Thus $\fa(S_2) - \fa(S_1) \fst^\ba \fb(S_2)-\fb(S_1)$
by Proposition~\ref{witless1}.

(b) Let $\g \in \GRID^\ebmu$ with $\g \fst^\ebmu 1$.
Then $\g = \g_1\g_2\cdots\g_J$,
where $\g_j \in \bmu$ for all $j$ and $J \ge 1$.
Now $\g_j \fst 1$, $\g_j > 0$ and $S_1 < S_2$, so
$0 < \g_j(S_2) < \g_j(S_1)$ and therefore
$\dom(\g_j(S_1)) \ge \dom(g_j(S_2))$ and
$\mag(\g_j(S_1)) \fgteq^\ba \mag(\g_j(S_2))$.  We consider two cases.

\textit{Case 1.} $\dom(g_j(S_1)) > \dom(g_j(S_2))$ for some $j$.
Then
$$
	\dom(\g(S_1)) = \prod_{j=1}^J \dom(\g_j(S_1)) >
	\prod_{j=1}^J \dom(\g_j(S_2)) = \dom(\g(S_2)) .
$$
So $\mag(\g(S_1)-\g(S_2)) = \mag(\g(S_1))$.
Now let $\m \in \supp(\g(S_1)-\g(S_2))$.  One possibility
is $\m \in \supp(\g(S_1))$, so $\m = \prod_{j=1}^J \m_j$
with $\m_j \in \supp(\g_j(S_1))$ for all $j$.
But since $\ba$ witnesses $\g_j(S_1)$, this means
$\m_j \fsteq^\ba \mag(\g_j(S_1))$.  Therefore
$\m = \prod \m_j \fsteq^\ba \prod\mag(\g_j(S_1)) = \mag(\g(S_1))
=\mag(\g(S_1)-\g(S_2))$.
The other possibility is $\m \in \supp(\g(S_2))$, so
$\m = \prod \m_j$ with $\m_j \in \supp(\g_j(S_2))$.  But
$\ba$ witnesses $\g_j(S_1)-\g_j(S_2)$ and $\g_j(S_2)$, so
$\m_j \fsteq^\ba \mag(\g_j(S_2)) \fsteq^\ba \mag(\g_j(S_1))$.
Then as before $\m \fsteq^\ba \mag(\g(S_1)-\g(S_2))$.
Therefore $\ba$ witnesses $\g(S_1)-\g(S_2)$.

\textit{Case 2.} $\dom(g_j(S_1)) = \dom(g_j(S_2))$ for all $j$.
Write $\g_j(S_2) = \g_j(S_1)\cdot(1-V_j)$
with $V_j \fst^\ba 1, V_j>0$.  Note $\ba$ witnesses
$V_j = (\g_j(S_1)-\g_j(S_2))/\g_j(S_1)$.
Then
$$
	1 - \prod_{j=1}^J (1-V_j) = \sum_{j=1}^J V_j + U,
$$
where each term of $U$ is $\fst^\ba$ one of the $V_j$
and $\mag \sum V_j$ is $\mag V_j$ for the largest of the $V_j$.
So $\ba$ witnesses
$1 - \prod (1-V_j)$.  Now
$\g(S_1)-\g(S_2) = \g(S_1)\cdot (1 - \prod (1-V_j))$
and $\ba$ also witnesses $\g(S_1)$, so
$\ba$ witnesses $\g(S_1)-\g(S_2)$.

(b$'$) Let $\g \in \GRID^\ebmu$, $\g \fgt^\ebmu 1$.  As already
noted, $\ba$ witnesses $\g(S_1)$ and $\g(S_2)$.  Now
$\g^{-1} \fst^\ebmu 1$, so we apply (b) to it:
$\ba$ witnesses $\g^{-1}(S_1) - \g^{-1}(S_2)$ and therefore
$\ba$ witnesses
$\g(S_2) - \g(S_1) = \g^{-1}(S_1)\g^{-1}(S_2)
\big(\g^{-1}(S_1) - \g^{-1}(S_2)\big)$.

(c) Let $B \in \TWG\ebmu\ebmu$.  Then $\fb := \mag B \in \GRID^\ebmu$.
Assume $\fb \ne 1$.
Let $B = \sum a_\m \m$.  If $a_\m \m$ is any term in $B$ other
than the dominant term, then since $\bmu$ witnesses $B$, we
have $\m \fst^\ebmu \fb$, and therefore
$\m(S_2)-\m(S_1) \fst^\ba \fb(S_2) - \fb(S_1)$
by (a).  So
\begin{align*}
	a_\m \m(S_1) - a_\m \m(S_1) &\fst^\ba \fb(S_2)-\fb(S_1)
	\quad\text{if $\m \fst \fb$,}
\tag{1}
	\\
	a_\m \m(S_1) - a_\m \m(S_1) &\fe \fb(S_2)-\fb(S_1)
	\quad\text{if $\m = \fb$.}
\end{align*}
Summing these, we get $B(S_2) - B(S_1) \fe \fb(S_2)-\fb(S_1)$.

(d) With the notation of (c), assume also $\fb \fst^\ba 1$.
Then sum (1) and note $\ba$ witnesses $\fb(S_2)-\fb(S_1)$ to
conclude that $\ba$ witnesses $B(S_2)-B(S_1)$.

(e) Let $A, B \in \TWG\ebmu\ebmu$, $A \fst^\ebmu B \fst^\ebmu 1$.
Write $\fb = \mag B$.  For every $\fa \in \supp A$ we have
$\fa \fst^\ba \fb$, so as in (c) we conclude
$A(S_2) - A(S_1) \fst^\ba B(S_2)-B(S_1)$.

(f) follows from (d) and (e).

(g) follows from (d) and (e).

\end{proof}

\begin{re}\label{case4counter}
In Case 4 in the proof for Theorem~\ref{MVT}:
Although $\mu_i(S_2)-\mu_i(S_1) \fst^\ba V-1$
and $V \fgt 1$, we cannot conclude
$\mu_i(S_2)-\mu_i(S_1) \fst^\ba V$.  In fact, we cannot choose ratio set $\ba$
that will achieve this.
For an example: let $\bmu = \{\mu_1,\mu_2\}$,
$\mu_1 = e^{-x}$, $\mu_2=e^{-e^x}$, $S_1=x$, $S_2=2x$.
Write $\nu_i = \mu_i(S_2)/\mu_i(S_1)$, so
$\nu_1 = e^{-x}$ and
$\nu_2 = e^{-e^{2x}+e^x}$ are small monomials and
$\nu_1^{j} \nu_2^{-1} \fgt 1$ for $j \in \N$.
Take $\bp=(j,-1)$ so $\bmu^\bp = \mu_1^{j} \mu_2^{-1}$,
$V = \bmu^\bp(S_2)/\bmu^\bp(S_1)$.
Assume
$\mu_1(S_2)-\mu_1(S_1) \fst^\ba V$ (for all $j$).
Compute
\begin{align*}
	\mu_1(S_2)-\mu_1(S_1) &= e^{-2x}-e^{-x} ,
	\\
	V=\frac{\bmu^\bp(S_2)}{\bmu^\bp(S_1)} &=
	\frac{\displaystyle e^{e^{2x}-2jx}}{\displaystyle e^{e^x-jx}}
	= e^{e^{2x}-e^x-jx},
\end{align*}
a monomial.  So we have $e^{-x} \fst^\ba e^{e^{2x}-e^x-jx}$ for  all $j$.
This means $\ba^*$ contains all
$e^{-e^{2x}+e^x+(j-1)x}$ and is therefore not well-ordered.
So we have a contradiction.
\end{re}

\begin{re}
The following is not true:  Given $\bmu$, $S_2$, $S_1$, there is $\ba$
so that: if $\g \in \GRID^\ebmu$ then $\ba$
witnesses $\g(S_2)-\g(S_1)$.  This is a continuation of
Remark~\ref{case4counter}.  Let $\bmu, S_1, S_2$ be as before.
Let $\g = \mu_1^j\mu_2^{-1}$. So
$\g(S_2) - \g(S_1) = e^{e^{2x}-2jx} - e^{e^x-jx}$.
If $\ba$ witnesses this, then
$e^{-e^{2x}+e^x+jx} \in \ba^*$.
As noted, this is not possible for all $j$ that this belong to the
same grid $\ba^*$.
\end{re}

\begin{re}
The following is not true:  Given $\bmu, S_1, S_2$, there is $\ba$
so that: if $\sum A_j$ converges $\bmu$-geometrically, then
$\sum(A_j(S_2) - A_j(S_1))$
converges $\ba$-geometrically.
This is another continuation of Remark~\ref{case4counter}.
Let $A_j = \mu_1^j\mu_2^{-1}$.  Then $\sum A_j$
converges $\bmu$-geometrically.  But there is no ratio set $\ba$
that witnesses all terms $A_j(S_2) - A_j(S_1)$.
[Does this suggest that we should we change the definition
of geometric convergence?]
\end{re}

Potentially, there is a separate theorem like Theorem~\ref{MVT}
for each $[\DD_n]$ in \cite[\S\Ctaylorlabel]{edgarc}.


\begin{thebibliography}{99}

\bibitem{asch}
M. Aschenbrenner, L. van den Dries,
``Asymptotic differential algebra.''
In~\cite{proc}, pp.~49--85

\bibitem{costintop}
O. Costin, ``Topological construction of transseries and introduction to generalized Borel summability.''
In~\cite{proc}, pp.~137--175

\bibitem{costinglobal}
O. Costin, ``Global reconstruction of analytic functions
from local expansions and a new general method of converting
sums into integrals.'' preprint, 2007.\hfill\break
\texttt{http://arxiv.org/abs/math/0612121}

\bibitem{costinasymptotics}
O. Costin,
\textit{Asymptotics and Borel Summability.}
CRC Press, London, 2009

\bibitem{proc}
O. Costin, M. D. Kruskal, A. Macintyre (eds.),
\emph{Analyzable Functions and Applications}
(\emph{Contemp. Math.} \textbf{373}).
Amer. Math. Soc., Providence RI, 2005

\bibitem{DMM}
L. van den Dries, A. Macintyre, D. Marker,
``Logarithmic-exponential series.''
\emph{Annals of Pure and Applied Logic} \textbf{111} (2001) 61--113

\bibitem{edgar}
G.~Edgar,
``Transseries for beginners.'' preprint, 2009.\hfill\break
\texttt{http://arxiv.org/abs/0801.4877} or\hfill\break
\texttt{http://www.math.ohio-state.edu/\hbox{$\sim$}edgar/preprints/trans\_begin/}

\bibitem{edgarc}
G.~Edgar,
``Transseries: composition, recursion, and convergence.''
\textit{forthcoming}\hfill\break
\texttt{http://arxiv.org/abs/0909.1259v1} or\hfill\break
\texttt{http://www.math.ohio-state.edu/\hbox{$\sim$}edgar/preprints/trans\_compo/}

\bibitem{gravett}
K. A. H. Gravett,
``Valued linear spaces.''
\emph{Quart. J. Math. Oxford} (2) \textbf{7} (1955) 309--315

\bibitem{higman}
G. Higman, ``Ordering by divisibility in abstract algebras.''
\textit{Proc. London Math. Soc.} \textbf{2} (1952) 326--336

\bibitem{hoeven}
J. van der Hoeven,
\emph{Transseries and Real Differential Algebra}
(\emph{Lecture Notes in Mathematics} \textbf{1888}).
Springer, New York, 2006

\end{thebibliography}
\end{document}